\documentclass[12pt]{amsart}

\usepackage{amssymb, amscd, txfonts}
\usepackage{graphicx}

\numberwithin{equation}{section}

\newtheorem{theorem}{Theorem}[section]
\newtheorem{proposition}[theorem]{Proposition}
\newtheorem{lemma}[theorem]{Lemma}
\newtheorem{corollary}[theorem]{Corollary}

\theoremstyle{definition}

\theoremstyle{remark}
\newtheorem{remark}[theorem]{Remark}

\newcommand{\N}{\mathbb{N}}
\newcommand{\Z}{\mathbb{Z}}
\newcommand{\Zp}{\mathbb{Z}_{p}}
\newcommand{\Zpuni}{\mathbb{Z}_{p}^{\times}}
\newcommand{\Q}{\mathbb{Q}}

\newcommand{\C}{\mathbb{C}}
\newcommand{\HH}{\mathbb{H}}
\newcommand{\proj}{{\mathbb P}}

\newcommand{\SL}{{\rm SL}_2(\mathbb{Z})}
\newcommand{\Mp}{{\rm Mp}_2(\mathbb{Z})}
\newcommand{\Or}{{\rm O}^+}
\newcommand{\Ost}{\widetilde{{\rm O}}^+}

\newcommand{\HM}{{\rm vol}_{HM}}

\newcommand{\HMOL}{{\rm vol}_{HM}({\rm O}^{+}(L))}
\newcommand{\HMOK}{{\rm vol}_{HM}({\rm O}^{+}(K))}
\newcommand{\DL}{\mathcal{D}_{L}}
\newcommand{\FL}{\mathcal{F}_{L}}
\newcommand{\FLcpt}{\bar{\mathcal{F}_{L}}}
\newcommand{\OAL}{{\rm O}(A_{L})}
\newcommand{\OAK}{{\rm O}(A_{K})}
\newcommand{\divi}{{\rm div}}
\newcommand{\ord}{{\rm ord}}
\newcommand{\SAI}{\mathcal{S}(\textrm{AI})}
\newcommand{\SBI}{\mathcal{S}(\textrm{BI})}
\newcommand{\SBII}{\mathcal{S}(\textrm{BII})}
\newcommand{\SBIII}{\mathcal{S}(\textrm{BIII})}

\begin{document}

\title[]{Finiteness of stable orthogonal modular varieties of non-general type}
\author[]{Shouhei Ma}
\thanks{Supported by Grant-in-Aid for Scientific Research No.12809324 and No.22224001.} 
\address{Department~of~Mathematics, Tokyo~Institute~of~Technology, Tokyo 152-8551, Japan}
\email{ma@math.titech.ac.jp}
\maketitle 

\begin{abstract}
We prove that there are only finitely many even lattices $L$ 
of signature $(2, n)$ with $n\geq15$ such that 
the modular variety defined by the stable orthogonal group of $L$ is not of general type. 
\end{abstract}


\section{Introduction}\label{sec:intro}

\subsection{Main result}

Our main object in this article is to show that 
a certain class of modular varieties canonically attached to even lattices $L$ of signature $(2, n)$ 
are almost always of general type, provided that $n$ is not in a low range. 
The modular groups we will work with are defined as follows. 
Let $L^{\vee}$ be the dual lattice of $L$ and $A_L=L^{\vee}/L$ be the discriminant group, 
which is canonically equipped with a ${\Q}/2{\Z}$-valued quadratic form.  
We have a natural homomorphism ${\rm O}(L)\to{\OAL}$, 
whose kernel is denoted by $\widetilde{{\rm O}}(L)$ and referred to as the \textit{stable orthogonal group} of $L$. 
This is a canonical congruence subgroup of ${\rm O}(L)$. 
Let ${\DL}$ be the Hermitian symmetric domain of type IV associated to $L$. 
It is either of the two components of the space 
\begin{equation*}\label{eqn:def type IV domain}
\{  {\C}\omega \in {\proj}(L\otimes{\C}) \; | \; (\omega, \omega)=0, (\omega, \bar{\omega})>0  \} . 
\end{equation*}
Then let ${\Ost}(L)$ be the subgroup of $\widetilde{{\rm O}}(L)$ preserving ${\DL}$. 
The quotient space 
\begin{equation*}\label{eqn:def FL}
{\FL} = {\Ost}(L)\backslash{\DL} 
\end{equation*}
has the structure of a quasi-projective variety of dimension $n$. 
We are interested in the birational type of ${\FL}$.

\begin{theorem}\label{main}
There are only finitely many even lattices $L$ of signature $(2, n)$ with $n\geq15$ 
such that ${\FL}$ is not of general type. 
In particular, ${\FL}$ is always of general type if $n$ is sufficiently large. 
\end{theorem}

This is essentially an effective result, in the sense that with a huge amount of computation 
one would be able to enumerate all possible $L$ for which ${\FL}$ \textit{may} not be of general type 
(cf.~Remarks \ref{eff calcu}, \ref{effective calcu cusp form} and \ref{effective bigness}). 
For instance, we can tell that ${\FL}$ is always of general type at least in the range $n\geq42$. 
It is another, nontrivial problem to indeed find non-general type examples. 
The traditional method is to look for a dominant period map from a uniruled parameter space; 
Gritsenko recently found an approach using modular forms (\cite{Gr2}, \cite{G-H}).

In complex algebraic geometry, modular varieties of orthogonal type  
appear as the period spaces of (lattice-)polarized $K3$ surfaces and hyperK\"ahler manifolds.  
For instance, when $L=2U\oplus 2E_8 \oplus \langle -2d\rangle$, 
${\FL}$ gives the moduli space of polarized $K3$ surfaces of degree $2d$. 
Gritsenko-Hulek-Sankaran \cite{G-H-S1} proved that this space is of general type for almost all $d$, 
while Mukai has shown that it is unirational for some small $d$ (see \cite{Mu} and the references there). 
Our result emphasizes the viewpoint that the latter spaces are exceptional regarding birational type. 
It also suggests that it will get hard to 
construct \textit{generic} polarized hyperK\"ahler manifolds \textit{explicitly}, 
when the polarized Beauville form $L$ has sufficiently large rank and 
its monodromy group is contained in ${\Ost}(L)$.

The phenomenon that even a natural class of modular varieties tend to be of general type in higher dimension 
was first discovered for the symplectic groups by Tai \cite{Ta}. 
Gritsenko-Hulek-Sankaran \cite{G-H-S3} found a similar picture with ${\FL}$ for a special series of lattices $L$.  
This was the first result on the Kodaira dimension of orthogonal modular varieties in higher dimension. 
The present article was inspired by their work.   
 

Ideally, it would be desirable to extend the finiteness result 
from the stable orthogonal groups ${\Ost}(L)$ to the full orthogonal groups ${\rm O}^+(L)$. 
It would be also not unreasonable to expect 
an analogue of Borisov's (\cite{Bori}) finiteness theorem for $2U\oplus \langle-2\rangle$ 
to hold for more general lattices. 

In the next \S \ref{ssec:reduction} and \S \ref{decomposition} 
we set up the proof of Theorem \ref{main}. 
We first reduce it to the proof of Theorem \ref{reduce to quasi-cyclic}, 
which in turn is decomposed into Theorems \ref{effective} and \ref{bigness}. 
The last two theorems are purely in terms of modular forms, 
and the bulk of the article is devoted to proving them.

\subsection{Reduction of lattices}\label{ssec:reduction}

For the proof of Theorem \ref{main} we shall reduce the class of lattices.  
We say that an even lattice $L$ is \textit{quasi-cyclic} if any isotropic subgroup of its discriminant form $A_L$ is cyclic. 
Geometrically this condition implies that any $1$-dimensional cusp of ${\FL}$ contains 
the standard $0$-dimensional cusp in its closure, 
which in turn assures that the Jacobi lifting of cusp forms are cusp forms (\cite{G-H-S1} \S 4). 
The notion of quasi-cyclic lattice is an extension of maximal lattice. 
We have studied their discriminant forms in the Appendix. 

We show that 
Theorem \ref{main} can be deduced from the following weaker assertion. 

\begin{theorem}\label{reduce to quasi-cyclic}
There are only finitely many quasi-cyclic lattices of signature $(2, n)$ with $n\geq15$ 
such that ${\FL}$ is not of general type. 
\end{theorem}

The connection with Theorem \ref{main} 
is provided by Proposition \ref{q-cyclic overlattice}, according to which 
for an arbitrary even lattice $L$ we can find a quasi-cyclic overlattice $L'\supset L$ such that 
$A_{L'}$ has exponent no smaller than half of $A_L$. 
We have the inclusion ${\Ost}(L)\subset{\Ost}(L')$ inside ${\rm O}(L\otimes{\Q})$, 
which induces a finite morphism ${\FL}\to\mathcal{F}_{L'}$. 
In particular, if $n$ is so large that $\mathcal{F}_{L'}$ is of general-type for any quasi-cyclic lattice $L'$ of signature $(2, n)$, 
then so is ${\FL}$ for any even lattice $L$ of the same signature. 
Next we fix a (not sufficiently large) $n\geq15$. 
There exists a natural number $D$ such that 
for quasi-cyclic lattices $L'$ of signature $(2, n)$, $\mathcal{F}_{L'}$ is of general type 
whenever the exponent of $A_{L'}$ exceeds $D$. 
Therefore, for general even lattices $L$ of signature $(2, n)$, 
$\mathcal{F}_{L}$ is of general type if the exponent of $A_L$ exceeds $2D$. 
Since there are only finitely many abelian groups of bounded exponent ($\leq2D$) and length ($\leq n+2$) 
and since there are only finitely many even lattices of fixed signature and discriminant group, 
the finiteness follows at this fixed $n$. 
Thus Theorem \ref{main} can be deduced from Theorem \ref{reduce to quasi-cyclic}.

In the next \S \ref{decomposition} we explain the plot of the proof of Theorem \ref{reduce to quasi-cyclic}.

\subsection{The two obstructions}\label{decomposition}

We shall reduce the proof of Theorem \ref{reduce to quasi-cyclic} to that of Theorems \ref{effective} and \ref{bigness}. 
These two sub-theorems will be proved in the rest of the article. 
We essentially follow the approach proposed by Gritsenko-Hulek-Sankaran \cite{G-H-S3}, 
which in \cite{G-H-S3} was used to prove general-typeness of ${\FL}$ for some special lattices $L$. 
What we eventually show is that in the range $n\geq15$ 
a generalization of their argument actually applies to all but finitely many quasi-cyclic lattices.

Let $L$ be an even lattice of signature $(2, n)$ with $n\geq9$ (which we do not assume to be quasi-cyclic for the moment). 
By \cite{G-H-S1} Theorem 2.1, we can find a projective toroidal compactification ${\FLcpt}$ of ${\FL}$ 
that has only canonical quotient singularities and 
has no branch divisor in the boundary. 
The second condition means that for each rational boundary component $F$ of ${\DL}$ 
the projection $X_{\Sigma}(F)\to{\FLcpt}$ from the torus embedding $X_{\Sigma}(F)$ 
is unramified at general points of the boundary divisors over $F$.

Let $\mathcal{L}$ be the ${\Q}$-line bundle over ${\FLcpt}$ of modular forms of weight $1$, 
$\Delta\subset{\FLcpt}$ the boundary divisor of the compactification, 
and $B\subset{\FLcpt}$ the branch divisor of the projection ${\DL}\to{\FL}$. 
According to \cite{G-H-S1} Theorem 2.1, every irreducible component of $B$ is defined by a reflection, 
and in particular has ramification index $2$. 
The canonical divisor of ${\FLcpt}$ is then ${\Q}$-linearly equivalent to (cf.~\cite{G-H-S1} \S 1)
\begin{equation*}\label{eqn:cano div}
K_{{\FLcpt}} \sim_{{\Q}} n\mathcal{L} - \Delta - B/2. 
\end{equation*}
The ${\Q}$-bundle $n\mathcal{L}$ is big 
as it is the pullback of an ample ${\Q}$-line bundle on the Baily-Borel compactification. 
It will be our source of deriving bigness of $K_{{\FLcpt}}$. 
We will view the remaining components $-\Delta$ and $-B/2$ as obstructions for $K_{{\FLcpt}}$ to be big,  
and deal with them separately by dividing the weight $n$ of $n\mathcal{L}$. 
To be precise, we will prove the following.

\begin{theorem}\label{effective}
There are only finitely many quasi-cyclic lattices $L$ of signature $(2, n)$ with $n\geq15$ such that 
the divisor $n'\mathcal{L}-\Delta$ of ${\FLcpt}$ is not effective for any $n'<n$. 
In other words, for all but finitely many quasi-cyclic $L$ with $n\geq15$ 
we can find a cusp form of weight $<n$ with respect to ${\Ost}(L)$.  
\end{theorem}

\begin{theorem}\label{bigness}
Fix a rational number $a>0$.  
Then there are only finitely many even lattices $L$ 
of signature $(2, n)$ and containing $2U$ 
such that the ${\Q}$-divisor $a\mathcal{L}-B/2$ of ${\FL}$ is not big.  
\end{theorem}

Note that in Theorem \ref{effective} the lattices are assumed to be quasi-cyclic, 
while in Theorem \ref{bigness} they are only assumed to contain $2U$. 
By Corollary \ref{q-cyclic contain 2U} quasi-cyclic lattices of signature $(2, n)$ with $n\geq8$ always contain $2U$.

Now Theorem \ref{reduce to quasi-cyclic} can be derived as follows. 
The combination of Theorem \ref{effective} and Theorem \ref{bigness} with $a=1$ 
implies that for all but finitely many quasi-cyclic lattices $L$ with $n\geq15$, 
we can find a division 
\begin{equation}\label{eqn:division}
n\mathcal{L} - \Delta - B/2 = (n'\mathcal{L}-\Delta) + (n''\mathcal{L}-B/2), \quad n'+n''=n, 
\end{equation}
such that $n'\mathcal{L}-\Delta$ is effective and that $n''\mathcal{L}-B/2$ is big. 
Hence $K_{{\FLcpt}}$ is big for those $L$. 
Since ${\FLcpt}$ has canonical singularity, a resolution of ${\FLcpt}$ is of general type. 
This proves Theorem \ref{reduce to quasi-cyclic}.

The proof of Theorem \ref{effective} will be given in \S \ref{sec:cusp form}, 
and that of Theorem \ref{bigness} will occupy \S \ref{sec:bigness setup} -- \S \ref{sec:vol ratio}. 

\begin{remark}\label{eff calcu}
Here is an effective process to see whether a given ${\FL}$ is of general type with the present method. 
(See the subsequent sections for detail, 
especially Remarks \ref{effective calcu cusp form} and \ref{effective bigness} for actual computation.) 
\begin{enumerate}
\item First find a weight $l$ of which there exists a cusp form of type $\rho_L$. 
\item We want to try $n'=n/2-1+l$ and $n''=n/2+1-l$ in \eqref{eqn:division}. 
So substitute $a=n/2+1-l$ in the right side of \eqref{eqn: big via HM volume}. 
\item Estimate the left side of \eqref{eqn: big via HM volume} 
using the results of \S \ref{sec:sum over branch div} and \S \ref{sec:vol ratio}, 
and check whether the inequality holds. 
\end{enumerate}
\end{remark}


\section{Construction of cusp forms}\label{sec:cusp form}

This section is devoted to the proof of Theorem \ref{effective}. 
In order to construct a cusp form,  
we use the method of Jacobi lifting by Gritsenko \cite{Gr}. 
Thus we need to produce Jacobi cusp forms. 
For that we translate Jacobi forms into vector-valued modular forms of one variable, 
and analyze a dimension estimate for the latter.

\subsection{Jacobi forms}\label{ssec:Jacobi form}

We begin by recalling Jacobi forms of $1+$several variables. 
Let $N=M(-1)$ be an arbitrary positive-definite even lattice, where $M$ is negative-definite. 
A Jacobi form of index $1$ and weight $k\in{\N}$ for $N$ is 
a holomorphic function $\phi(\tau, Z)$ on ${\HH}\times(N\otimes{\C})$ 
satisfying the transformation rules 
\begin{equation*}
\phi(\tau, \, Z+l\tau+m) = e^{-\pi i(l, l)\tau-2\pi i(l, Z)}\phi(\tau, Z) 
\end{equation*}
for $l, m\in N$ and   
\begin{equation*}\label{eqn:Jacobi trans rule I}
\phi\left(\gamma\tau, \, \frac{Z}{c\tau+d}\right) = 
(c\tau+d)^k{\exp}\left( \frac{\pi ic(Z, Z)}{c\tau+d}\right) \phi(\tau, Z)   
\end{equation*}
for $\gamma = \begin{pmatrix} a & b \\ c & d \end{pmatrix} \in {\SL}$, 
and having Fourier expansion of the form 
\begin{equation*}
\phi(\tau, Z) 
= \sum_{\begin{subarray}{c} n\in{\N},\, l\in N^{\vee} \\ (l, l)\leq2n \end{subarray}} 
c(n, l)e^{2\pi in\tau+2\pi i(l, Z)}. 
\end{equation*} 
If $c(n, l)=0$ for any $(n, l)$ with $(l, l)=2n$, $\phi$ is called a Jacobi cusp form. 
We denote by $J_{k,1}(N)$ the space of Jacobi forms of weight $k$ and index $1$ for $N=M(-1)$. 

The connection with modular forms of orthogonal type is given by 
the following Jacobi lifting. 

\begin{theorem}[Gritsenko \cite{Gr}, \cite{G-H-S1}]\label{Jacobi lifting}
For $k\geq4$ there exists an injective linear map 
from $J_{k,1}(M(-1))$ to the space of modular forms of weight $k$ with respect to ${\Ost}(2U\oplus M)$. 
If the lattice $M$ is quasi-cyclic, 
this maps Jacobi cusp forms to cusp forms. 
\end{theorem}

The lifting can be defined even when $k<4$ provided that the Fourier coefficient $c(0, 0)$ vanishes, 
but we do not need that. 
See \cite{Gr} Theorem 3.1 for the explicit form of lifting. 
The assertion on cusp forms was originally proved in \cite{Gr} 
under the assumption that $M$ is maximal; 
later it was extended in \cite{G-H-S1} Theorem 4.2 to the present version. 

Let ${\Mp}$ be the metaplectic double cover of ${\SL}$ and 
\begin{equation*}
\rho_M\colon{\Mp}\to{\rm U}({\C}[A_M])
\end{equation*} 
be the Weil representation attached to $M$, for which we follow the convention of \cite{Bo}. 
Jacobi forms correspond to modular forms of type $\rho_M$ as follows (cf.~\cite{Gr}). 
For each $\lambda\in A_N$ we consider the theta function 
\begin{equation*}
\theta_{N}^{\lambda}(\tau, Z) = \sum_{l\in N+\lambda} e^{\pi i(l, l)\tau+2\pi i(l, Z)}. 
\end{equation*}
Then a Jacobi form $\phi\in J_{k,1}(N)$ can be uniquely expanded as 
\begin{equation*}
\phi(\tau, Z) = \sum_{\lambda\in A_N} \phi_{\lambda}(\tau)\theta_{N}^{\lambda}(\tau, Z)
\end{equation*}
for some ${\C}[A_N]$-valued function 
$\Phi(\tau) = (\phi_{\lambda}(\tau))_{\lambda\in A_N}$ on ${\HH}$. 
We shall identify $A_N$ with $A_M=A_{N(-1)}$ naturally and view $\Phi(\tau)$ as ${\C}[A_M]$-valued. 
Note that $\rho_M=\rho_N^{\vee}$ under this identification. 
Then the transformation formula of $(\theta_N^{\lambda})_{\lambda}$ under ${\Mp}$ 
tells that $\Phi(\tau)$ is a modular form of weight $k-{\rm rk}(M)/2$ and type $\rho_M$ for ${\Mp}$. 

For a half integer $l>0$ we denote by $M_l(\rho_M)$ 
the space of modular forms of weight $l$ and type $\rho_M$ for ${\Mp}$, 
and $S_{l}(\rho_M)\subset M_l(\rho_M)$ the subspace of cusp forms. 

\begin{proposition}[cf.~\cite{Gr}]\label{Jacobi form = rhoM-type form}
The correspondence 
$\sum_{\lambda}\phi_{\lambda}\theta_{N}^{\lambda} \mapsto (\phi_{\lambda})_{\lambda}$ 
defines an isomorphism 
\begin{equation*}
J_{k,1}(M(-1)) \simeq M_{k-{\rm rk}(M)/2}(\rho_M), 
\end{equation*}
which preserves the subspaces of cusp forms. 
\end{proposition}

A dimension formula for $S_l(\rho_M)$ is presented in \cite{Sk}, \cite{E-S} and \cite{Bo}. 
Below we follow the version in \cite{Bo}. 
Define elements $T, S, Z\in{\Mp}$ by 
\begin{equation*}
T=\left( \begin{pmatrix}1&1\\ 0&1\end{pmatrix}, 1 \right), \quad  
S=\left( \begin{pmatrix}0&-1\\ 1&0\end{pmatrix}, \sqrt{\tau} \right), \quad  
Z=\left( \begin{pmatrix}-1&0\\ 0&-1\end{pmatrix}, \sqrt{-1} \right). 
\end{equation*}
Let $\{ e_{\lambda} \}_{\lambda\in A_M}$ be the standard basis of ${\C}[A_M]$, 
and $W_+(M), W_-(M)\subset{\C}[A_M]$ be the subspaces spanned by vectors of the form 
$e_{\lambda}+e_{-\lambda}$,   
$e_{\lambda}-e_{-\lambda}$ respectively. 
The decomposition 
\begin{equation*}\label{eqn:Z-decomp}
{\C}[A_M] = W_+(M) \oplus W_-(M)   
\end{equation*}
equals to the eigendecomposition for $\rho_M(Z)$ with 
$\rho_M(Z)|_{W_{\pm}(M)}=\pm\sqrt{-1}^{{\rm rk}(M)}$. 
Since $Z$ is in the center of ${\Mp}$, 
this is actually a decomposition as an ${\Mp}$-representation. 
We shall denote 
\begin{equation*}
d_{\pm}(M) = {\dim}W_{\pm}(M), \qquad 
\rho_{M}^{\pm} = \rho_M|_{W_{\pm}(M)}. 
\end{equation*}
The automorphic condition for $Z$ implies that 
$M_l(\rho_M)$ can be nonzero only when $l+{\rm rk}(M)/2 \in {\Z}$. 
More precisely, modular forms in $M_l(\rho_M)$ take values 
in $W_+(M)$ if $l+{\rm rk}(M)/2$ is even, 
and in $W_-(M)$ if $l+{\rm rk}(M)/2$ is odd. 
To state the dimension formula, 
for any unitary matrix $X$ of size $d$ with eigenvalues 
$e(\beta_1), \cdots, e(\beta_d)$ with $0\leq\beta_i<1$ 
where $e(x)=e^{2\pi ix}$, 
we write 
\begin{equation*}
\alpha(X) = \sum_{i=1}^{d} \beta_i. 
\end{equation*}
Then we set 
\begin{eqnarray*}
\alpha_1^{\pm}(M) &=& \alpha(e(l/4)\rho_{M}^{\pm}(S)), \\ 
\alpha_2^{\pm}(M) &=& \alpha(e(-l/6)\rho_{M}^{\pm}(ST)^{-1}), \\ 
\alpha_3^{\pm}(M) &=& \alpha(\rho_{M}^{\pm}(T)), \\   
\alpha_4^+(M) &=& \# (\{ \lambda\in A_M \: | \: (\lambda, \lambda)\in 2{\Z} \}/\pm1), \\ 
\alpha_4^-(M) &=& \# (\{ \lambda\in A_M \: | \: (\lambda, \lambda)\in 2{\Z}, \: 2\lambda\ne0  \}/\pm1).    
\end{eqnarray*}

\begin{proposition}[\cite{Sk}, \cite{E-S}, \cite{Bo}]\label{dim cusp form}
The dimension of $S_l(\rho_M)$ is bounded by  
\begin{equation}\label{eqn:dim formula Sl(rhoM)}
{\dim}\: S_l(\rho_M) \geq  
d_{\pm}(M) + d_{\pm}(M)\cdot l/12 - \sum_{i=1}^{4}\alpha_i^{\pm}(M), 
\end{equation}
where $\pm$ is chosen according to the parity of $l+{\rm rk}(M)/2\in{\Z}$. 
This becomes equality when $l>2$. 
\end{proposition}

In \cite{Br} Bruinier estimated $\alpha_1^+(M)$ and $\alpha_2^+(M)$. 
(Note that $\rho_L$ in the convention of \cite{Br} is rather $\rho_L^{\vee}=\rho_{L(-1)}$ in that of \cite{Bo}.)  
A similar estimate is also possible for $\alpha_1^-(M)$ and $\alpha_2^-(M)$, 
and we can state the result in the following uniform manner.

\begin{lemma}[cf.~\cite{Br}]\label{Bruinier estimate}
Let $G_2, G_3\subset A_M$ denote the subgroups of elements $x$ with $2x=0$, $3x=0$ respectively. 
Then we have  
\begin{eqnarray*}
\alpha_1^{\pm}(M) &\leq & \frac{d_{\pm}(M)}{4} + \frac{1}{4}\sqrt{|G_2|}, \\
\alpha_2^{\pm}(M) &\leq & \frac{d_{\pm}(M)}{3} + \frac{1}{3\sqrt{3}}\left( 1+\sqrt{|G_3|}\right), 
\end{eqnarray*}
where $\pm$ is according to the parity of $l+{\rm rk}(M)/2$. 
\end{lemma}

\begin{proof}
See \cite{Br} Corollary 3 for the $W_+(M)$-valued case. 
In the $W_-(M)$-valued case, 
we can still follow the argument of \cite{Br} Lemma 2 to deduce 
\begin{eqnarray*}
\alpha_1^-(M) & = & \frac{d_-(M)}{4} + 
\frac{e(x_1)}{4\sqrt{|A_M|}}{\rm Im}(G(2, M)), \\
\alpha_2^-(M) & = & \frac{d_-(M)}{3} + 
\frac{1}{3\sqrt{3|A_M|}}{\rm Re}[ e(x_2) (G(1, M)-G(-3, M)) ], 
\end{eqnarray*}
where 
$x_1=(2l+{\rm rk}(M)+2)/8$, 
$x_2=(-4l-3{\rm rk}(M)+10)/24$, 
and $G(d, M)$ is the Gauss sum 
$\sum_{\lambda\in A_M}e(d(\lambda, \lambda)/2)$.  
Then we can use \cite{Br} Lemma 1 to estimate the Gauss sums.   
\end{proof}

Bruinier also estimated $\alpha_3^+(M)$, but we will not use that.

\subsection{Proof of Theorem \ref{effective}}\label{ssec:prf eff}

We are now ready to prove Theorem \ref{effective}. 
Let $L$ be a quasi-cyclic lattice of signature $(2, n)$ with $n\geq8$. 
By Corollary \ref{q-cyclic contain 2U} the lattice $L$ contains $2U$: we write it in the form 
\begin{equation*}
L = 2U \oplus M 
\end{equation*}
with $M$ quasi-cyclic and negative-definite of rank $n-2$. 
We shall apply Proposition \ref{dim cusp form} to this $M$.  
By Corollaries \ref{q-cyclic length p>2} and \ref{q-cyclic length p=2} we have  
\begin{equation}\label{eqn:estimate G2 G3}
|G_2| \leq 2^5, \qquad |G_3|\leq 3^4. 
\end{equation} 
Incorporating this with Lemma \ref{Bruinier estimate} gives estimates of $\alpha_1^{\pm}(M)$ and $\alpha_2^{\pm}(M)$. 
For the $\alpha_3$ and $\alpha_4$-terms, we content ourselves with the trivial inequality 
\begin{equation}\label{eqn:trivial bound alpha3+4}
\alpha_3^{\pm}(M) + \alpha_4^{\pm}(M) \leq d_{\pm}(M). 
\end{equation}
If we put 
$\varepsilon = \sqrt{2} + 10/\sqrt{27}$, 
we obtain in this way  
\begin{equation*}
{\dim}\: S_l(\rho_M) \geq d_{\pm}(M) \cdot \frac{l-7}{12} - \varepsilon, 
\end{equation*}
where $l$ is a half integer with $l+n/2-1\in{\Z}$ and $\pm$ is according to the parity of $l+n/2-1$. 
When the right hand side is positive, the Jacobi lifting produces a cusp form for ${\Ost}(L)$ because $M$ is quasi-cyclic. 
Therefore 

\begin{proposition}\label{estimate for cusp form exist}
Let $L$ be a quasi-cyclic lattice of signature $(2, n)$ with $n\geq8$. 
Let $l$ be a half integer with $l+n/2\in{\Z}$. 
If 
\begin{equation*}
d_{\pm}(L)\cdot (l-7) > 12\varepsilon, 
\end{equation*} 
where $\pm$ is according to the parity of $l+n/2-1$, 
then there exists a cusp form of weight $l+n/2-1$ with respect to ${\Ost}(L)$. 
\end{proposition}

Theorem \ref{effective} follows from this proposition and the following lemma. 

\begin{lemma}\label{auxiliary for finiteness}
There are only finitely many quasi-cyclic lattices $L$ of signature $(2, n)$ with $n\geq15$ 
which does not admit half integer $l\equiv n/2$ mod ${\Z}$ satisfying both 
\begin{equation*}
d_{\pm}(L)\cdot(l-7) > 12\varepsilon, \qquad l< n/2+1. 
\end{equation*}
\end{lemma}

\begin{proof}
We first prove the finiteness at each \textit{fixed} $n$. 
Take $l=15/2$ when $n$ is odd and $l=8$ when even. 
This satisfies $l<n/2+1$ and $l>7$. 
It then suffices to see that the inequality 
$d_{\pm}(L) \leq 12\varepsilon/(l-7)$ 
holds for only finitely many $L$, where $\pm$ is determined by $[n]\in{\Z}/4{\Z}$. 
Since 
$|A_L|\leq2d_{\pm}(L)+2^5$ 
holds by \eqref{eqn:estimate G2 G3}, 
there are only finitely many quasi-cyclic discriminant forms with $d_{\pm}$ bounded. 
With the signature $(2, n)$ fixed, a quasi-cyclic lattice is determined by its discriminant form 
according to \cite{Ni} and Corollary \ref{q-cyclic contain 2U}. 

Next, if $n$ is sufficiently large, we can choose $l$ so that 
\begin{equation*}
l+n/2-1 \in 2{\Z}, \qquad 
l< n/2+1, \qquad 
l>12\varepsilon+7. 
\end{equation*}
(For instance, when $n\geq97$, we may take $l=49-n_0/2$ where $n_0\equiv n$ mod $4$ with $0\leq n_0\leq3$.)  
In this case $+$ is chosen, and  
\begin{equation*}
d_+(L) \geq 1 > 12\varepsilon/(l-7) 
\end{equation*}
holds for any $L$. 
\end{proof}


\begin{remark}\label{effective calcu cusp form}
In the proof of Lemma \ref{auxiliary for finiteness} we gave a bound $n\geq97$ 
where the assertion of Theorem \ref{effective} holds with no exception, but this is too coarse. 
We can improve the estimate of $l$ by separating the small $d_{\pm}(L)$ case 
and calculating \eqref{eqn:dim formula Sl(rhoM)} directly there. 
At least for maximal $L$ with $|A_L|>2$, 
we can thus see that there always exists a cusp form of type $\rho_L$ and weight $\leq8$. 
Ideally, it is desirable to improve the poor estimate \eqref{eqn:trivial bound alpha3+4}. 
\end{remark}


\section{The reflective obstruction}\label{sec:bigness setup}

In this section we shall carry out the proof of Theorem \ref{bigness}, 
with the proof of two crucial inequalities postponed to later sections. 
In \S \ref{ssec:branch div} we recall the classification of the branch divisors following \cite{G-H-S1}. 
In \S \ref{ssec:HM vol} we show that the bigness of $a\mathcal{L}-B/2$ can be derived from 
an inequality involving Hirzebruch-Mumford volumes. 
Then in \S \ref{ssec:proof bigness} we state (without proof) the key estimates, 
and deduce Theorem \ref{bigness} from them. 
By thus separating the actual calculations, 
we hope to clarify the structure of the proof of Theorem \ref{bigness}.

\subsection{The branch divisors}\label{ssec:branch div}

The basis of our argument is Gritsenko-Hulek-Sankaran's classification \cite{G-H-S1} of 
irreducible components of the branch divisor of ${\FL}$. 
Let $L$ be an even lattice of signature $(2, n)$ with $n>2$. 
Recall that the reflection $\sigma_l$ with respect to a primitive vector $l\in L$ of norm $<0$ 
is defined by the equation 
\begin{equation*}
\sigma_l : L\otimes{\Q} \to L\otimes{\Q}, \quad v\mapsto v-\frac{2(v, l)}{(l, l)} l. 
\end{equation*}
When $\sigma_l \in {\Or}(L)$ and $\sigma_l \equiv \pm{\rm id}$ on $A_L$, 
then $l$ is called a \textit{stably reflective vector}. 
According to \cite{G-H-S1} Corollary 2.13, 
irreducible components of the ramification divisor of ${\DL}\to{\FL}$ are 
exactly the fixed divisors of $\sigma_l$ of stably reflective vectors $l$. 
If $K=l^{\perp}\cap L$, this fixed divisor is nothing but the hyperplane section 
\begin{equation*}
\mathcal{D}_{K} = {\proj}(K\otimes {\C})\cap{\DL}. 
\end{equation*}
Thus the classification of irreducible components of the branch divisor of ${\FL}$ is equivalent to 
that of stably reflective vectors up to $\langle {\Ost}(L), -1\rangle$. 

In \cite{G-H-S1} \S 3, stably reflective vectors are initially classified as follows. 
We denote ${\divi}(l)=[{\Q}l\cap L^{\vee}:{\Z}l]$, which is the natural number generating $(l, L)$. 
Note that $\sigma_l$ preserves $L$ exactly when $(l, l)=-{\divi}(l)$ or $-2{\divi}(l)$.

\begin{proposition}[\cite{G-H-S1}]\label{classify stab ref vect} 
Let $l\in L$ be a stably reflective vector. 

\noindent
(A) If  $\sigma_l \equiv {\rm id}$ on $A_L$, then either 

(AI) $(l, l)=-2$ and ${\divi}(l)=2$;  

(AII) $(l, l)=-2$ and ${\divi}(l)=1$. 

\noindent
(B) Suppose that $\sigma_l \equiv -{\rm id}$ but $\nequiv {\rm id}$ on $A_L$. 
Then $A_L$ must be of the form $A_L\simeq {\Z}/D\oplus({\Z}/2)^m$ for some $D>2$ and $m\geq0$, 
with either  

(BI) $(l, l)=-D$ and ${\divi}(l)=D$;  

(BII) $(l, l)=-D$ and ${\divi}(l)=D/2$; 

(BIII) $(l, l)=-2D$, ${\divi}(l)=D$ with $D$ odd and $m=0$. 
\end{proposition}

The condition $m=0$ in case (BIII) is not stated in \cite{G-H-S1}, 
but this follows at once because $D$ was originally defined in \cite{G-H-S1} as the exponent of $A_L$. 
We required $\sigma_l\nequiv{\rm id}$ in case (B) in order to avoid overlap with case (A): 
this forces $D>2$. 
Notice that $D$ must be even in cases (BI) and (BII).



\subsection{Hirzebruch-Mumford volume}\label{ssec:HM vol}

For a while let $L_0$ be an arbitrary integral lattice of signature $(2, n_0)$ with $n_0>0$ 
and $\Gamma\subset{\Or}(L_0)$ be a finite-index subgroup. 
We write $M_k(\Gamma)$ for the space of modular forms of weight $k$ with respect to $\Gamma$. 
Gritsenko-Hulek-Sankaran \cite{G-H-S2} defined 
the Hirzebruch-Mumford volume ${\HM}(\Gamma)$ of $\Gamma$ 
following the proportionality principle of Hirzebruch and Mumford. 
It determines the growth behavior of the dimension of $M_k(\Gamma)$ by 
\begin{equation}\label{eqn:HM volume and modular forms}
{\rm dim}\, M_k(\Gamma) = \frac{2}{n_0!}{\HM}(\Gamma) k^{n_0} + O(k^{n_0-1}), 
\end{equation}
where we restrict to even $k$ if $-1\in\Gamma$. 
We may adopt this as an equivalent definition of ${\HM}(\Gamma)$. 
If $\Gamma'\subset\Gamma$ is a cofinite subgroup, we have 
\begin{equation}\label{eqn:HM cofinite subgrp}
{\HM}(\Gamma') = [\langle \Gamma, -1\rangle : \langle \Gamma', -1\rangle]\cdot {\HM}(\Gamma). 
\end{equation}

Now let $L$ be an even lattice of signature $(2, n)$ 
for which we are studying whether the ${\Q}$-divisor $a\mathcal{L}-B/2$ of ${\FL}$ is big where $a\in{\Q}$. 
This problem can be related to the Hirzebruch-Mumford volumes of ${\Ost}(L)$ and of the branch divisors. 
Choose representatives $l_1,\cdots, l_r\in L$ of 
the set of $\langle {\Ost}(L), -1\rangle$-equivalence classes of stably reflective vectors. 
Let $K_i=l_i^{\perp}\cap L$, and 
$\Gamma_i\subset{\Or}(K_i)$ be the image of the stabilizer of ${\Z}l_i$ in ${\Ost}(L)$. 
Then $\Gamma_i\backslash\mathcal{D}_{K_i}$ is the normalization of the component of $B$ defined by $l_i$.

\begin{proposition}\label{big via HM volume}
For $a\in {\Q}_{>0}$ the ${\Q}$-divisor $a\mathcal{L}-B/2$ of ${\FL}$  is big if we have 
\begin{equation}\label{eqn: big via HM volume}
\sum_{i=1}^{r} \frac{{\HM}(\Gamma_i)}{{\HM}({\Ost}(L))}
< \left( 1+\frac{1}{a} \right)^{1-n} \cdot \frac{2a}{n}. 
\end{equation}
\end{proposition}

\begin{proof}
The ${\Q}$-divisor $a\mathcal{L}-B/2$ is big if the asymptotic estimate 
\begin{equation}\label{eqn:bigness estimate}
h^0(ka\mathcal{L}-(k/2)B) > c\cdot k^n  
\end{equation}
holds for some $c>0$, 
where $k$ grows under the condition that both $k$ and $ka$ are even numbers. 
We shall bound the left hand side. 
The key is the following estimate, which is essentially proved in \cite{G-H-S3} Proposition 4.1.  

\begin{lemma}[cf.~\cite{G-H-S3}]\label{quasi-pullback estimate}
When both $k$ and $ka$ are even numbers, we have 
\begin{equation}\label{eqn:quasipullback} 
h^0(ka\mathcal{L}-(k/2)B) \geq {\dim}\, M_{ka}({\Ost}(L)) - \sum_{i=1}^{r}\sum_{j=0}^{k/2-1}{\dim}\, M_{ka+2j}(\Gamma_i). 
\end{equation}
\end{lemma}

\begin{proof}
We present a proof for the convenience of the reader. 
If $j$ is a natural number, 
$H^0(ka\mathcal{L}-jB)$ is identified with the space of ${\Ost}(L)$-modular forms of weight $ka$ 
which have zero of order $\geq 2j$ along every $\mathcal{D}_{K_i}$. 
The quasi-pullback to $\mathcal{D}_{K_i}$ of such modular forms is defined by 
\begin{equation*}\label{eqn:quasi-pullback}
H^0(ka\mathcal{L}-jB) \to M_{ka+2j}(\Gamma_i), \qquad F\mapsto (F/(\cdot, l_i)^{2j})|_{\mathcal{D}_{K_i}}. 
\end{equation*}
Note that $F$ must have zero of even order along $\mathcal{D}_{K_i}$, by its invariance under $-\sigma_{l_i}$. 
Therefore we obtain the exact sequence  
\begin{equation*}
0 \to H^0(ka\mathcal{L}-(j+1)B) \to H^0(ka\mathcal{L}-jB) \to \bigoplus_{i=1}^{r}M_{ka+2j}(\Gamma_i). 
\end{equation*}
Iteration of this for $j=0,\cdots, k/2-1$ gives the desired estimate. 
\end{proof}

We study the asymptotic behavior of the right side of \eqref{eqn:quasipullback} 
with respect to $k$. 
For the first term, we have by \eqref{eqn:HM volume and modular forms}  
\begin{equation*}
{\dim}\, M_{ka}({\Ost}(L)) = (2/n!) \cdot {\HM}({\Ost}(L)) \cdot a^n \cdot k^n + O(k^{n-1}). 
\end{equation*}
The second term is estimated as 
\begin{eqnarray*}
& & \sum_{i=1}^{r}\sum_{j=0}^{k/2-1}{\dim}\, M_{ka+2j}(\Gamma_i)  \\ 
&=& \sum_{i=1}^{r}\sum_{j=0}^{k/2-1}\left\{ \frac{2}{(n-1)!}\cdot{\HM}(\Gamma_i)\cdot 
            (ka+2j)^{n-1}+O(k^{n-2})\right\}  \\ 
&\leq& \sum_{i=1}^{r}\frac{k}{2}\cdot\left\{ \frac{2}{(n-1)!}\cdot{\HM}(\Gamma_i)\cdot 
           (a+1)^{n-1}\cdot k^{n-1}+O(k^{n-2})\right\}.  
\end{eqnarray*}
Comparison of the coefficients of $k^n$ in these two asymptotics gives the condition \eqref{eqn: big via HM volume}. 
\end{proof}



\subsection{Proof of Theorem \ref{bigness}}\label{ssec:proof bigness}

Now we state key estimates, \eqref{eqn:sum over branch div} and \eqref{eqn:vol ratio}, 
and deduce Theorem \ref{bigness} from them. 
From now on we assume that 
\begin{equation}\label{condition:L=2U+M}
\text{$L$ is an even lattice of signature $(2, n)$ and contains $2U$.} 
\end{equation}
 
We shall be more specific in \eqref{eqn: big via HM volume}. 
For each type $*=AI,\cdots, BIII$ as in Proposition \ref{classify stab ref vect}, 
we write $\mathcal{R}(*)$ for the set of $\langle {\Ost}(L), -1\rangle$-equivalence classes of 
stably reflective vectors of type $*$. 
Then the left side of \eqref{eqn: big via HM volume} can be rewritten as 
\begin{equation}\label{eqn:(3.3)}
\sum_{*} \sum_{\mathcal{R}(*)} \frac{{\HM}(\Gamma)}{{\HM}({\Ost}(L))}
\end{equation}
where $K=l^{\perp}\cap L$ for $[l]\in R(*)$ and 
$\Gamma\subset{\Or}(K)$ is the image of the stabilizer of ${\Z}l$ in ${\Ost}(L)$. 
Since $\Gamma\supset{\Ost}(K)$, the formula \eqref{eqn:HM cofinite subgrp} shows that 
\begin{eqnarray*}
\eqref{eqn:(3.3)} 
& \leq & 
\sum_{*} \sum_{\mathcal{R}(*)} \frac{{\HM}({\Ost}(K))}{{\HM}({\Ost}(L))} \\ 
& = & 
\sum_{*} \sum_{\mathcal{R}(*)} \frac{|{\OAK}/\pm1|}{|{\OAL}/\pm1|} \cdot \frac{{\HM}({\Or}(K))}{{\HM}({\Or}(L))}.  
\end{eqnarray*}
Here the homomorphisms ${\Or}(L)\to{\OAL}$ and ${\Or}(K)\to{\OAK}$ are surjective by \cite{Ni}, 
because $L$ and $K$ contain $U$ (see Lemma \ref{K contains U} for $K$).

In the remaining sections we will prove the following. 

\begin{proposition}\label{sum over branch div}
For each type $*$ of stably reflective vector, we define the function $e_{\ast}(n)$ by 
the right end of Table \ref{summary table}. 
Then for any even lattice $L$ as in \eqref{condition:L=2U+M} we have 
\begin{equation}\label{eqn:sum over branch div}
\sum_{\mathcal{R}(*)} \frac{|{\OAK}/\pm1|}{|{\OAL}/\pm1|} \leq e_{\ast}(n). 
\end{equation}
\end{proposition}

\begin{proposition}\label{vol ratio}
For each type $\ast$ of stably reflective vector, 
we define the function $f_{\ast}(n)$ 
by Proposition \ref{vol estimate A} and Corollary \ref{estimate B by n AL}. 
Then for any even lattice $L$ as in \eqref{condition:L=2U+M} and stably reflective vector $l\in L$ of type $\ast$ 
with orthogonal complement $K=l^{\perp}\cap L$, 
we have 
\begin{equation}\label{eqn:vol ratio}
\frac{{\HM}({\Or}(K))}{{\HM}({\Or}(L))} < f_{\ast}(n)\cdot |A_L|^{-1/2}. 
\end{equation}
\end{proposition}

Assuming \eqref{eqn:sum over branch div} and \eqref{eqn:vol ratio}, 
we obtain the estimate 
\begin{equation*}
\eqref{eqn:(3.3)} \; < \; |A_L|^{-1/2} \cdot \sum_{\ast} e_{\ast}(n) f_{\ast}(n). 
\end{equation*}
Incorporating this with Proposition \ref{big via HM volume}, 
we see that $a\mathcal{L}-B/2$ is big whenever 
\begin{equation}\label{eqn:final estimate for bigness}
\sqrt{|A_L|}  \; \geq \;  \frac{n}{2a} \cdot \left( 1+\frac{1}{a} \right)^{n-1} \cdot \sum_{\ast} e_{\ast}(n) f_{\ast}(n).  
\end{equation}
This first implies the finiteness at each fixed $n$, 
because we have only finitely many even lattices $L$ of fixed signature and with $|A_L|$ bounded. 
Next, the explicit form of $e_{\ast}(n)$ and $f_{\ast}(n)$ 
tells that the right side of \eqref{eqn:final estimate for bigness} converges to $0$ as $n\to\infty$. 
Hence \eqref{eqn:final estimate for bigness} holds for any $L$ if $n$ gets sufficiently large. 
This proves Theorem \ref{bigness}.

The proof of Propositions \ref{sum over branch div} and \ref{vol ratio} shall be carried out respectively in 
\S \ref{sec:sum over branch div} and \S \ref{sec:vol ratio}, through case-by-case estimate. 
The calculation is certainly lengthy, but explicit. 
Especially, it tells how to improve the bounding function $\sum_{\ast}e_{\ast}(n)f_{\ast}(n)$ 
when the class of lattices is specified, 
which will be important for making Theorem \ref{bigness} as effective as possible.

 
\section{Proof of Proposition \ref{sum over branch div}}\label{sec:sum over branch div}

Throughout this section $L$ is assumed to be an even lattice as in \eqref{condition:L=2U+M}. 
For each type of stably reflective vectors $l$, 
we will describe the discriminant form $A_K$ of the orthogonal complement  
\begin{equation*}
K=l^{\perp}\cap L, 
\end{equation*}
and then derive estimates as in Proposition \ref{sum over branch div}. 
Let $\mathcal{S}(*)$ denote the set of ${\Ost}(L)$-equivalence classes of stably reflective vectors of type $*=AI,\cdots, BIII$. 
Though our original target is to estimate
\begin{equation}\label{eqn:stab ortho ratio +-1}
\sum_{\mathcal{S}(*)/\pm1} |{\OAK}/\pm1| / |{\OAL}/\pm1|, 
\end{equation} 
it is more convenient to deal with 
\begin{equation}\label{eqn:stably reflect ratio}
\sum_{\mathcal{S}(*)} |{\OAK}| / |{\OAL}|. 
\end{equation}
We will estimate \eqref{eqn:stably reflect ratio} first; 
the relation with \eqref{eqn:stab ortho ratio +-1} 
is clarified in Lemma \ref{eqn:estimate after /-1}. 
The results of this section will be finally summarized in Table \ref{summary table}.

\subsection{Preliminaries}

Let us prepare some tools and notations. 
We write  
\begin{equation*}
A_L = \bigoplus_{p} (A_L)_p 
\end{equation*} 
for the decomposition into $p$-components. 
If we denote ${\rm O}(A_L)_p = {\rm O}((A_L)_p)$, 
we have the canonical decomposition 
\begin{equation*}
{\OAL} = \prod_{p} {\rm O}(A_L)_p. 
\end{equation*} 

When studying $A_K$, the following technique (\cite{Ni}) will be used. 

\begin{lemma}\label{Nikulin gluing}
Suppose that $L$ contains ${\Z}l\oplus K$ by index $2$. 
Then there exists an element $\lambda\in A_K$ of order $2$ and norm $-(l, l)/4$ mod $2{\Z}$ such that 
the subgroup $L/({\Z}l\oplus K)$ of $A_{{\Z}l\oplus K}$ is given by the element $\lambda'=[l/2]+\lambda$. 
Furthermore, $A_L$ is isometric to the quadratic form induced on $(\lambda')^{\perp}/\lambda'$. 
\end{lemma}  
 
See \cite{Ni} \S 1.5 for the full account.  
We will call $\lambda\in A_K$ the \textit{gluing element} for the extension $L\supset{\Z}l\oplus K$.

For estimating \eqref{eqn:stably reflect ratio} the following \textit{Eichler criterion} will be essential 
(see, e.g., \cite{Sc} Section 3.7). 
It is primarily this proposition where we need the assumption that $L$ contains $2U$. 
For a primitive vector $l\in L$ we denote $l^{\ast}=l/{\divi}(l)\in L^{\vee}$. 

\begin{proposition}[Eichler criterion]\label{Eichler criterion} 
The ${\Ost}(L)$-equivalence class of a primitive vector $l$ of $L$ is 
determined by the norm $(l, l)$ and the element $[l^{\ast}]\in A_L$. 
\end{proposition}

This will be applied with the following supplement, which is a trivial generalization of \cite{Sc} Lemma 4.1.1. 

\begin{lemma}\label{Eichler supple}
Let $x\in A_L$ be an element of order $d$ and norm $\equiv r$ mod $2{\Z}$ with $r\in{\Q}$. 
Writing $L=2U\oplus M$, 
we can find a primitive vector $l$ from the sublattice $U\oplus M$ such that 
${\divi}(l)=d$, $(l, l)=d^2r$ and $[l^{\ast}]=x\in A_L$. 
\end{lemma}

\begin{proof}
Since $A_M=A_L$, we can take a vector $m\in M^{\vee}$ with $[m]=x\in A_L$. 
Since $(m, m)\equiv r$ mod $2{\Z}$, 
we may write $r-(m, m)=2k$ for some integer $k$. 
Let $e, f$ be the standard hyperbolic basis of $U$. 
Then the vector 
\begin{equation*}
l = d(e+kf+m) \in U\oplus M^{\vee}
\end{equation*}
is contained in $U\oplus M$ and has norm $d^2r$. 
It is primitive in $L$ because $d'm\notin M$ for any $0<d'<d$. 
Since ${\divi}(d(e+kf))=d$, we see that ${\divi}(l)=d$. 
Therefore $[l^{\ast}]=[m]=x$. 
\end{proof}

With this lemma, Eichler's criterion implies the following. 

\begin{proposition}\label{Eichler+}
The set of ${\Ost}(L)$-equivalence classes of primitive vectors $l\in L$ with 
$(l, l)=N$ and ${\divi}(l)=d$ can be identified with the set 
\begin{equation*}
\{ x\in A_L \: | \: {\ord}(x)=d, \: (x, x) \equiv d^{-2}N \: {\rm mod} \: 2{\Z} \}, 
\end{equation*}
by the correspondence $l\mapsto[l^{\ast}]=[l/d]\in A_L$. 
\end{proposition}

We can also deduce the following.  

\begin{lemma}\label{K contains U}
The orthogonal complement $K=l^{\perp}\cap L$ of any primitive vector $l\in L$ contains $U$. 
\end{lemma}

\begin{proof}
By the Eichler criterion and Lemma \ref{Eichler supple}, 
$l$ is ${\Ost}(L)$-equivalent to a primitive vector $l'$ in $U\oplus M\subset L$. 
Clearly $(l')^{\perp}\cap L$ contains $U$. 
\end{proof}

We are now ready to start our analysis.

\subsection{Type AI}\label{ssec:AI} 

Let $l\in L$ be a vector with $(l, l)=-2$ and ${\divi}(l)=2$, which is always stably reflective. 
Since $(l, l)=-{\divi}(l)$, we have the splitting 
\begin{equation*}
L = {\Z}l \oplus K \simeq \langle -2 \rangle \oplus K. 
\end{equation*}
Accordingly, if we denote $b=A_{\langle -2 \rangle}$, we have 
$A_L \simeq b \oplus A_K$. 

By Proposition \ref{Eichler+}, 
the set of ${\Ost}(L)$-equivalence classes of vectors $l\in L$ with $(l, l)=-2$ and ${\divi}(l)=2$ 
can be identified with the set 
\begin{equation*}
{\SAI} = \{ x\in (A_L)_2 \: | \: {\ord}(x)=2, (x, x)\equiv -1/2 \}, 
\end{equation*} 
through the correspondence $l\mapsto[l/2]\in A_L$.  

\begin{lemma}\label{ratio estimate AI}
We have $\sum_{{\SAI}}|{\OAK}|/|{\OAL}| \leq 9$. 
\end{lemma}

\begin{proof}
If we consider the action of ${\OAL}$ on ${\SAI}$, 
then ${\OAK}$ is regarded as the stabilizer subgroup for $x=[l/2]$. 
It follows that 
\begin{equation*}
\sum_{{\SAI}}|{\OAK}|/|{\OAL}| = | {\SAI}/{\OAL} | = | {\SAI}/{\rm O}(A_L)_2 |. 
\end{equation*}
By the correspondence $x\mapsto x^{\perp}\cap(A_L)_2$, 
we can identify ${\SAI}/{\rm O}(A_L)_2$ with the set of isometry classes of 
finite quadratic forms $A$ with $b\oplus A \simeq (A_L)_2$. 
To estimate its cardinality, 
we use Kawauchi-Kojima's (\cite{K-K}) invariants $\sigma_r$ 
of nondegenerate quadratic forms on $2$-groups\footnote{
Here we identify, as in \cite{Wa} Theorem 5, 
quadratic forms and symmetric bilinear forms with no direct summand of order $2$.},  
which are defined for each $r\geq1$. 
They take values in the semigroup $({\Z}/8{\Z})\cup\{\infty\}$, and have the properties that  
(i) two such forms $A$, $A'$ are isometric if and only if 
$A\simeq A'$ as abelian groups and $\sigma_r(A)=\sigma_r(A')$ for every $r\geq1$, 
and that (ii) $\sigma_r(A\oplus A')=\sigma_r(A)+\sigma_r(A')$. 
Now if $b\oplus A\simeq (A_L)_2$, 
the structure of $A$ as an abelian group is uniquely determined. 
Since $\sigma_r(b)\ne\infty$ for $r\ne2$, 
the value of $\sigma_r(A)$ is uniquely determined except $r=2$. 
The value of $\sigma_2(A)$ is taken from $({\Z}/8{\Z})\cup\{\infty\}$, 
so that we have $|{\SAI}/{\rm O}(A_L)_2|\leq9$. 
\end{proof}

\subsection{Type AII}\label{ssec:AII} 

Let $l\in L$ be a vector with $(l, l)=-2$ and ${\divi}(l)=1$, which is always stably reflective. 
Since $(l, l)=-2{\divi}(l)$, $L$ is an overlattice of ${\Z}l\oplus K\simeq \langle-2\rangle\oplus K$ of index $2$. 
Let us denote $a=A_{\langle2\rangle}$. 

\begin{lemma}\label{A_K AII}
We have $A_K\simeq a\oplus A_L$. 
\end{lemma}

\begin{proof}
Let $\lambda\in A_K$ be the gluing element as in Lemma \ref{Nikulin gluing}. 
Since $(\lambda, \lambda)\equiv 1/2$, 
$\langle \lambda \rangle$ is isometric to $a$ and in particular nondegenerate. 
Thus we have the orthogonal splitting $A_K=\langle\lambda\rangle\oplus A$ where $A=\lambda^{\perp}\cap A_K$. 
By the last assertion of Lemma \ref{Nikulin gluing} we see that $A\simeq A_L$. 
\end{proof}

The Eichler criterion tells that there is only one ${\Ost}(L)$-equivalence class of vectors $l\in L$ with 
$(l, l)=-2$ and ${\divi}(l)=1$. 
Then  

\begin{lemma}\label{ratio estimate AII}
We have $|{\OAK}|/|{\OAL}|\leq 2^{n-2}$. 
\end{lemma}
 
\begin{proof}
By Lemma \ref{A_K AII} we can view ${\OAL}$ as a subgroup of ${\OAK}$, 
namely the stabilizer group of the gluing element $\lambda$. 
Hence $|{\OAK}|/|{\OAL}|$ is equal to $\#({\OAK}\cdot \lambda)$,  
which in turn is dominated by   
\begin{equation*}
\# \{ x\in A_K \: | \: 2x=0, \: (x, x)\equiv 1/2 \} \leq  2^{l((A_K)_2)-1}. 
\end{equation*} 
Then we have 
\begin{equation}\label{eqn:overestimate 5}
l((A_K)_2) = l((A_L)_2)+1 \leq n-1. 
\end{equation}
\end{proof}

\subsection{Type BI}\label{ssec:BI} 

Let $D$ be an even number with $D\geq4$ and suppose that $A_L\simeq {\Z}/D\oplus({\Z}/2)^m$. 
Let $l\in L$ be a primitive vector with $(l, l)=-D$ and ${\divi}(l)=D$. 
Then we have the splitting 
\begin{equation*}
L = {\Z}l \oplus K \simeq \langle -D \rangle \oplus K, 
\end{equation*}
and hence 
$A_L \simeq \langle [l/D] \rangle \oplus A_K$. 
This implies that 
$A_K \simeq ({\Z}/2)^m$. 
In particular, $l$ is always stably reflective (see also \cite{G-H-S1} Proposition 3.2 (iii)). 

The set of ${\Ost}(L)$-equivalence classes of primitive vectors $l\in L$ with $(l, l)=-D$ and ${\divi}(l)=D$ 
can be identified with  
\begin{equation*}
{\SBI} = \{ x\in A_L \: | \: {\ord}(x)=D, (x, x)\equiv -1/D \}, 
\end{equation*} 
by associating $l\mapsto[l/D]\in A_L$.  
We factor $D$ as $D=2^{\nu}\cdot D_{\ne2}$ with $D_{\ne2}$ odd. 

\begin{lemma}\label{ratio estimate BI}
We have  
\begin{equation*} 
\sum_{{\SBI}}|{\OAK}|/|{\OAL}| \leq 
\begin{cases}
 9,  & \nu=1, \\
 1,  & \nu>1. 
\end{cases}
\end{equation*}
\end{lemma}

\begin{proof} 
As in the proof of Lemma \ref{ratio estimate AI}, 
we can view ${\OAK}$ as a stabilizer subgroup of ${\OAL}$ in its action on ${\SBI}$. 
Hence 
\begin{equation*}
\sum_{{\SBI}}|{\OAK}|/|{\OAL}| = |{\SBI}/{\OAL}|. 
\end{equation*}
To study ${\SBI}/{\OAL}$, we decompose $-1/D\in {\Q}/2{\Z}$ as 
\begin{equation*}
-1/D \equiv \alpha/2^{\nu} + \beta/D_{\ne2}, \quad \alpha\in{\Z}, \; \beta\in2{\Z}. 
\end{equation*}
If we set  
\begin{equation*}
\mathcal{S}_{\ne2} = \{ x\in \oplus_{p>2}(A_L)_p \: | \: {\ord}(x)=D_{\ne2}, (x, x)\equiv \beta/D_{\ne2} \}, 
\end{equation*} 
\begin{equation*}
\mathcal{S}_{2} = \{ x\in (A_L)_2 \: | \: {\ord}(x)=2^{\nu}, (x, x)\equiv \alpha/2^{\nu} \}, 
\end{equation*} 
we have the canonical decomposition 
\begin{equation*}
{\SBI} = \mathcal{S}_2 \times \mathcal{S}_{\ne2}. 
\end{equation*}
Since $(A_L)_p$ is cyclic for $p>2$, decomposing $\mathcal{S}_{\ne2}$ into $p$-parts shows that 
$\prod_{p>2}{\OAL}_p$ 
acts on $\mathcal{S}_{\ne2}$ transitively (and freely). 
Therefore 
\begin{equation*}
 {\SBI}/{\OAL} \simeq \mathcal{S}_2/{\OAL}_2.
\end{equation*}
As before, 
we can estimate $\#(\mathcal{S}_2/{\OAL}_2)$ by considering Kawauchi-Kojima's invariants $\sigma_r$ 
of $x^{\perp}\cap(A_L)_2$ for $x\in\mathcal{S}_2$. 
Since $x^{\perp}\cap(A_L)_2 \simeq ({\Z}/2)^m$, they already vanish for $r\ne2$. 
\end{proof}

\subsection{Type BII}\label{ssec:BII} 

Let $D\geq4$ be an even number and suppose that $A_L\simeq {\Z}/D\oplus({\Z}/2)^m$. 
We write $D=2^{\nu}\cdot D_{\ne2}$ with $D_{\ne2}$ odd. 
Let $l\in L$ be a primitive vector with $(l, l)=-D$ and ${\divi}(l)=D/2$. 
The lattice $L$ contains ${\Z}l\oplus K\simeq \langle-D\rangle \oplus K$ by index $2$. 
We denote by $\lambda\in A_K$ the gluing element as in Lemma \ref{Nikulin gluing}, 
which has norm $-2^{\nu-2}\cdot D_{\ne2}$ mod $2{\Z}$.

\begin{lemma}\label{A_K BII}
For $l\in L$ as above the following conditions are equivalent. 
\begin{enumerate} 
\item $l$ is stably reflective; 
\item $[2l/D]\in A_L$ is divisible by $2$ in $A_L$; 
\item $A_K\simeq ({\Z}/2)^{m+2}$.  
\end{enumerate} 
\end{lemma}

\begin{proof}
It is easy to see that these conditions always hold when $\nu=1$ (cf.~\cite{G-H-S1} Proposition 3.2 (iii)). 
In this case we have 
\begin{equation}\label{eqn:AK BII nu1}
A_K\simeq \langle \lambda \rangle \oplus (A_L)_2 
\end{equation}
as in Lemma \ref{A_K AII}. 
Below we let $\nu>1$, where $\lambda$ has norm $\in{\Z}$.

$(1) \Rightarrow (2)$: 
Since $A_K$ is nondegenerate, we can find an element $\mu\in A_K$ with $(\lambda, \mu)\equiv1/2$ mod ${\Z}$. 
Then $(l/D+\mu, l/2+\lambda)\equiv 0$ mod ${\Z}$, 
so that the element $[l/D]+\mu$ of $A_{{\Z}l\oplus K}$ gives that of $A_L$ by Lemma \ref{Nikulin gluing}. 
It suffices to show that $\mu$ has order $2$. 
The reflection $\sigma_l$ maps $[l/D]+\mu$ to $-[l/D]+\mu$. 
On the other hand, since $\sigma_l$ should act on $A_L$ by $-1$, we have 
\begin{equation*}
-[l/D]+\mu \equiv -[l/D]-\mu \; \in A_L. 
\end{equation*}
This shows that $2\mu=0$ in $A_K$. 

$(2) \Rightarrow (3)$: 
By assumption we have an element $y$ of $A_{{\Z}l\oplus K}$ with 
$(y, l/2+\lambda)=0$ and $2y=[2l/D]$ in $A_{{\Z}l\oplus K}$. 
(The possibility $2y=[2l/D]+[l/2]+\lambda$ can be excluded.) 
We may assume that $y$ is written as $[l/D]+\mu$ with $\mu\in A_K$. 
Then $2\mu=0$ and $(\mu, \lambda)=1/2$. 
Hence $\langle \lambda, \mu \rangle \simeq ({\Z}/2)^2$ is nondegenerate  
and we have the splitting $A_K=\langle \lambda, \mu \rangle\oplus A$ where $A=\langle \lambda, \mu \rangle^{\perp}$. 
Lemma \ref{Nikulin gluing} tells that $A_L\simeq \langle [l/D] + \mu \rangle \oplus A$. 
Hence $A\simeq({\Z}/2)^m$. 

$(3) \Rightarrow (1)$: 
Since $\sigma_l$ acts on $A_{{\Z}l\oplus K}$ by $(-1, 1)=(-1, -1)$, 
it acts on $A_L$ also by $-1$. 
\end{proof}

By this lemma, the set of ${\Ost}(L)$-equivalence classes of 
stably reflective vectors $l$ with $(l, l)=-D$ and ${\divi}(l)=D/2$ can be identified with the set 
\begin{equation*}
{\SBII} = \{ x\in A_L \: | \: {\ord}(x)=D/2, \; (x, x)\equiv -4/D, \; \textrm{divisible by} \: 2 \}, 
\end{equation*} 
by associating $l\mapsto[2l/D]\in A_L$. 

\begin{lemma}\label{estimate ratio BII}
We have 
\begin{equation*} 
\sum_{{\SBII}}|{\OAK}|/|{\OAL}| \leq 
\begin{cases}
2^{m+1},   & \nu=1, \\
2^{m+8},  & \nu=2, \\ 
2^{m+7},  & \nu>2. 
\end{cases}
\end{equation*}
\end{lemma}

\begin{proof}
We first consider the case $\nu=1$ where ${\SBII}$ is rewritten as 
\begin{equation*}
{\SBII} = \{ x\in \oplus_{p>2}(A_L)_p \: | \: {\ord}(x)=D_{\ne2}, \; (x, x)\equiv -2/D_{\ne2} \}. 
\end{equation*}
Since $(A_L)_p$ is cyclic for $p>2$, we can show that $\#{\SBII}$ is equal to $\prod_{p>2}|{\rm O}(A_L)_p|$ as before. 
By the relation \eqref{eqn:AK BII nu1} the isometry class of $A_K$ is determined independently of $x\in{\SBII}$, 
and we may view ${\OAL}_2$ as the stabilizer group of $\lambda$ in ${\OAK}$. 
Therefore 
\begin{equation*} 
\sum_{{\SBII}}|{\OAK}|/|{\OAL}| = |{\OAK}|/|{\OAL}_2| =  \# ({\OAK}\cdot\lambda ). 
\end{equation*}
It is then easy to estimate  
\begin{equation*}
 \# ({\OAK}\cdot\lambda)  
 \leq  \# \{ x\in A_K \: |\: (x, x)\equiv D_{\ne2}/2 \}  
 \leq  2^{m+1}. 
\end{equation*}

Next let $\nu>1$. 
For $x=[2l/D]$ we consider a decomposition $A_K=\langle \lambda, \mu \rangle\oplus A$ as in the proof of Lemma \ref{A_K BII}. 
Recall that $A_L=\langle [l/D]+\mu \rangle \oplus A$ and that $x$ is contained in $\langle [l/D]+\mu \rangle$. 
So if we let ${\OAL}$ act on ${\SBII}$, the stabilizer $G_x$ of $x$ contains ${\rm O}(A)$ with index 
\begin{eqnarray*}
\# (G_x\cdot([l/D]+\mu)) 
&   =   & \# \{ y \in A_L \: | \: 2y=x, \: (y, y)\equiv -1/D+(\mu, \mu) \} \\ 
& \geq & \# \{ y' \in A \: | \: (y', y')\equiv 0 \} \\ 
& \geq & 2^{m-3}.   
\end{eqnarray*} 
In particular, 
\begin{equation*}
\sum_{{\SBII}}|{\OAK}|/|{\OAL}| \leq \sum_{{\SBII}/{\OAL}}2^{3-m}\cdot|{\OAK}|/|{\rm O}(A)|. 
\end{equation*}
The index $|{\OAK}/{\rm O}(A)|$ is in turn bounded by the number of possible (abstract) embeddings 
$\langle \lambda, \mu \rangle \to A_K$. 
Therefore 
\begin{equation*}
2^{3-m}\cdot|{\OAK}/{\rm O}(A)| 
 \leq  2^{3-m}\cdot 2^{m+2}\cdot 2^{m+1}. 
\end{equation*}
In order to estimate $|{\SBII}/{\OAL}|$,  
we decompose $-4/D\in{\Q}/2{\Z}$ as 
$-4/D \equiv \alpha/2^{\nu-2} + \beta/D_{\ne2}$ 
with $\alpha\in{\Z}$, $\beta \in 2{\Z}$, 
and put accordingly 
\begin{equation*}
\mathcal{S}_2 = \{ x\in (A_L)_2 \: | \: {\ord}(x)=2^{\nu-1}, \; (x, x)\equiv \alpha/2^{\nu-2}, \; \textrm{divisible by} \: 2 \}. 
\end{equation*}
As in the proof of Lemma \ref{ratio estimate BI}, we obtain the reduction   
\begin{equation*}
{\SBII}/{\OAL} \simeq \mathcal{S}_2/{\OAL}_2. 
\end{equation*}
For each element of $\mathcal{S}_2/{\OAL}_2$ 
we \textit{choose} a representative $x\in\mathcal{S}_2$ 
and a subgroup $A'\subset (A_L)_2$ isomorphic to ${\Z}/2^{\nu}$ such that $2A'=\langle x \rangle$. 
The isometry class of $A'$ determines the ${\OAL}_2$-equivalence class of $\langle x\rangle$ 
because it uniquely determines the isometry class of $(A')^{\perp}\cap(A_L)_2\simeq ({\Z}/2)^m$ 
via Kawauchi-Kojima's invariants. 
Then we note the following. 
\begin{itemize}
\item Up to $\pm1$ we have at most $2$ (resp.~$1$) elements in $\langle x \rangle$ 
        of the same order and norm with $x$, in case $\nu\geq4$ (resp. $\nu\leq3$).  
\item We have at most $1$ (resp.~$2$, $4$) isometry classes of quadratic form $A'$ on ${\Z}/2^{\nu}$ such that 
         $2A'\simeq \langle x\rangle$ for $x\in\mathcal{S}_2$, in case $\nu\geq4$ (resp.~$\nu=3, 2$).  
\end{itemize}
These imply that  
$|\mathcal{S}_2/{\OAL}_2| \leq 4$ when $\nu=2$, and 
$|\mathcal{S}_2/{\OAL}_2| \leq 2$ when $\nu>2$. 
This finishes the proof of Lemma \ref{estimate ratio BII}. 
\end{proof}

\subsection{Type BIII}\label{ssec:BIII} 

Let $D\geq3$ be an odd number and suppose that $A_L\simeq{\Z}/D$. 
Note that $n$ must be even in this case, 
for $L\otimes{\Z}_2$ is an even unimodular ${\Z}_2$-lattice which should have even rank. 
A primitive vector $l\in L$ with $(l, l)=-2D$ and ${\divi}(l)=D$ 
is always stably reflective because $[l/D]$ generates $A_L$ (see also \cite{G-H-S1} Proposition 3.2 (iv)). 
Since $L$ is an index $2$ overlattice of ${\Z}l\oplus K$, 
we have $2^2\cdot|A_L|=2D\cdot|A_K|$.  
Therefore $A_K\simeq{\Z}/2$.

The set of ${\Ost}(L)$-equivalence classes of primitive vectors $l\in L$ with 
$(l, l)=-2D$ and ${\divi}(l)=D$ is identified with   
\begin{equation*}
{\SBIII} = \{ x\in A_L \: | \: {\ord}(x)=D, (x, x)\equiv -2/D \},  
\end{equation*} 
by mapping $l\mapsto[l/D]\in A_L$.  
If $\omega(D)$ is the number of prime divisors of $D$, 
both $\#{\SBIII}$ and  $|{\OAL}|$ are equal to $2^{\omega(D)}$. 
Since ${\OAK}=\{1\}$, then   

\begin{lemma}\label{estimate ratio BIII}
We have $\sum_{{\SBIII}}|{\OAK}|/|{\OAL}| =1$. 
\end{lemma}

\subsection{Summary}\label{ssec:summary table}

We have now estimated \eqref{eqn:stably reflect ratio} for each type of stably reflective vectors. 
We shall translate it to the estimate of \eqref{eqn:stab ortho ratio +-1}. 

\begin{lemma}\label{eqn:estimate after /-1}
For each type $\ast=AI,\cdots, BIII$ of stably reflective vectors, we have 
\begin{equation}\label{eqn:ratio +-1}
\sum_{\mathcal{S}(\ast)/\pm1} \frac{|{\OAK}/\pm1|}{|{\OAL}/\pm1|} = 
\delta \cdot \sum_{\mathcal{S}(\ast)} \frac{|{\OAK}|}{|{\OAL}|}, 
\end{equation}
where $\delta=2$ if $\ast=$ BII with $D=4$, and $\delta=1$ otherwise. 
\end{lemma}

\begin{proof}
In case $\ast=$ AI, AII, the $(-1)$-action on $\mathcal{S}(\ast)$ is trivial. 
Moreover, 
$A_K$ is 2-elementary if and only if $A_L$ is so. 
For $\ast=$ BI, BII, BIII, the $(-1)$-action on $\mathcal{S}(\ast)$ is free unless $\ast=$ BII and $D=4$. 
The group $A_K$ is always 2-elementary, while $A_L$ is never so. 
\end{proof}

For the convenience of Proposition \ref{sum over branch div}, we substitute $m\leq n-3$ in Lemma \ref{estimate ratio BII}. 
Then we can summarize (and simplify) the results in the following Table \ref{summary table}. 

\begin{table}[h]
\caption{Invariants of stably reflective vector}\label{summary table}
\begin{center} 
\begin{tabular}{c|c|c|c|c|c}
 & $(l, l)$ & ${\divi}(l)$ & $A_L$ & $A_K$ & the number \eqref{eqn:ratio +-1} \\ \hline  
AI    &  $-2$   & $2$ & --- & $b^{\perp}\cap A_L$ & $\leq 9$  \\ \hline  
AII   &  $-2$   & $1$ & --- & $a\oplus A_L$ & $\leq 2^{n-2}$  \\ \hline  
BI    &  $-D$   & $D$ & ${\Z}/D\oplus({\Z}/2)^m$ & $({\Z}/2)^m$ & $\leq 9$  \\ \hline  
BII   &  $-D$   & $D/2$ & ${\Z}/D\oplus({\Z}/2)^m$ & $({\Z}/2)^{m+2}$ & $\leq 2^{n+6}$ \\ \hline  
BIII  & $-2D$  & $D$  (odd) & ${\Z}/D$ & ${\Z}/2$ & $=1$ 
\end{tabular}
\end{center}
\end{table}


\section{Proof of Proposition \ref{vol ratio}}\label{sec:vol ratio}

In this section we estimate the ratio of the Hirzebruch-Mumford volumes 
\begin{equation*}\label{eqn:def vol ratio}
{\HM}(L, K) := \frac{{\HMOK}}{{\HMOL}} 
\end{equation*}
for the orthogonal complements $K=l^{\perp}\cap L$ of stably reflective vectors $l\in L$.

\subsection{A volume formula}\label{ssec:HM volume}

In this subsection, we let $L$ be an even lattice of signature $(2, n)$ and containing $U$. 
An explicit formula of ${\HMOL}$ is given in \cite{G-H-S2} Theorem 3.1 in terms of the local densities of $L$. 
Following a lot of examples in \cite{G-H-S2} \S 3, the formula was further developed in \cite{Ma}. 
Below we recall the version of \cite{Ma}.

There is a lot of notation originating from the local density formula in \cite{Ki} \S 5.6. 
For $p\geq2$ we denote by  
\begin{equation*}
L\otimes{\Zp} = \bigoplus_{j} L_{p,j}(p^j) 
\end{equation*} 
a Jordan decomposition of the ${\Zp}$-lattice $L\otimes{\Zp}$  
where $L_{p,j}$ is unimodular of rank $n_{p,j}\geq0$. 
We especially abbreviate $L_p=L_{p,0}$ and $n_p=n_{p,0}$. 
Let $s_p$ be the number of indices $j$ with $L_{p,j}\ne0$, and set  
\begin{equation*}
w_p = \sum_{j} jn_{p,j}\Bigl( \frac{n_{p,j}+1}{2} + \sum_{k>j} n_{p,k} \Bigr).  
\end{equation*}
For an even unimodular ${\Zp}$-lattice $M$ of rank $r\geq0$, 
we define $\chi(M)$ by 
$\chi(M)=0$ if $r$ is odd, 
$\chi(M)=1$ if $M\simeq (r/2)U\otimes{\Zp}$, 
and $\chi(M)=-1$ otherwise. 
For a natural number $l>0$ we put  
\begin{equation*}
P_p(l) = \prod_{k=1}^{l} (1-p^{-2k}), 
\end{equation*}
and $P_p(0)=1$. 

We need further notation for $p=2$. 
Consider a decomposition $L_{2,j}=L_{2,j}^{+}\oplus L_{2,j}^{-}$ such that 
$L_{2,j}^{+}$ is even and $L_{2,j}^{-}$ is either $0$ or odd of rank $\leq2$. 
Put $n_{2,j}^{+}={\rm rk}(L_{2,j}^{+})$. 
We set $q=\sum_j q_j$, where 
$q_j=0$ if $L_{2,j}$ is even, 
$q_j=n_{2,j}$ if $L_{2,j}$ is odd and $L_{2,j+1}$ is even, 
and $q_j=n_{2,j}+1$ if both $L_{2,j}$ and $L_{2,j+1}$ are odd. 
Here zero-lattice is counted as an even lattice. 
For those $j$ with $L_{2,j}\ne0$,  
we define $E_{2,j}(L)$ by  
$E_{2,j}(L) = 1+\chi(L_{2,j}^+)2^{-n_{2,j}^{+}/2}$ if both $L_{2,j-1}$ and $L_{2,j+1}$ are even 
and $L_{2,j}^-\nsimeq\langle\epsilon_1, \epsilon_2\rangle$ with $\epsilon_1\equiv\epsilon_2$ mod $4$, 
and $E_{2,j}(L)=1$ otherwise. 
We also let $s_2'$ be the number of indices $j$ such that 
$L_{2,j}=0$ and either $L_{2,j-1}$ or $L_{2,j+1}$ is odd.

In order to state the volume formula, 
we define four finite products, $F(L)$, $G(L)$, $H(L)$ and $C(L)$. 
Firstly $F(L)$ and $G(L)$ are defined by 
\begin{equation*}\label{eqn:F-term}
F(L) = \prod_{\begin{subarray}{c} p||A_L| \\ [n_p/2]\leq [n/2] \end{subarray}} 
\prod_{k=[n_p/2]+1}^{[n/2]+1} (1-p^{-2k}), 
\end{equation*}
\begin{equation*}\label{eqn:G-term}
G(L) = \prod_{p||A_L|, p>2} (1+\chi(L_p)p^{-n_p/2}) \cdot 
\begin{cases} 
1               & |A_L|:  \text{odd}, \\ 
E_{2,0}(L) & |A_L|:  \text{even}. 
\end{cases} 
\end{equation*}
The product $H(L)$ is defined when $n$ is even as follows. 
We can factorize $(-1)^{n/2+1}{\det}(L)$ as 
\begin{equation*}\label{eqn:fundamental discriminant}
(-1)^{n/2+1}{\det}(L) = t^2D 
\end{equation*}
with $D$ a fundamental discriminant\footnote{
This $D$ appears only in \S \ref{ssec:HM volume} and \S \ref{ssec:initial estimate}. 
No confusion is likely to occur with the exponent of $A_L$ as in Proposition \ref{classify stab ref vect} case (B).  
}. 
Denote by $\chi_{D}$ the Kronecker symbol $\Bigl( \frac{D}{\cdot} \Bigr)$. 
Then we set  
\begin{equation*}\label{eqn:Hm}
H(L) = G(L) \cdot \prod_{p||A_L|}\frac{1-\chi_{D}(p)p^{-n/2-1}}{1-p^{-n-2}}. 
\end{equation*}
Finally, for each prime factor $p$ of $|A_L|$ we put 
\begin{equation*}
C_p(L) = 
\begin{cases}
2^{1-s_p}  p^{-w_p}  \prod_{j}P_p([n_{p,j}/2])^{-1}(1+\chi(L_{p,j})p^{-n_{p,j}/2})
& p>2, \\
2^{1-s_2-s_2'-w_2+q} \prod_{j}P_2(n_{2,j}^+/2)^{-1} E_{2,j}(L)
& p=2, 
\end{cases}
\end{equation*}
where $j$ ranges over indices with $j>0$ and $L_{p,j}\ne0$. 
It will be convenient to set $C_p(L)=1$ even for $p\nmid |A_L|$. 
Then we define 
\begin{equation*}\label{eqn:C-term}
C(L) = 
\begin{cases}
8 \cdot (|A_L|/4)^{(n+3)/2}\cdot \prod_{p}C_p(L) 
& \text{$n$: odd,} \\
8 \sqrt{\pi} \cdot (|A_L|/4\pi)^{(n+3)/2}\cdot \prod_{p}C_p(L) 
& \text{$n$: even.} 
\end{cases}
\end{equation*}
Loosely speaking, 
$F(L)$ encodes only the rank of the unimodular components $L_p$, 
$G(L)$ and $H(L)$ encode the isometry classes of $L_p$, 
and $C(L)$ encodes information on the non-unimodular components $L_{p,j}$, $j>0$. 

We can now state the formula of ${\HMOL}$. 

\begin{proposition}[\cite{G-H-S2}, \cite{Ma}]\label{HM formula}
Let $L$ be an even lattice of signature $(2, n)$ and containing $U$. 

\noindent
(1) When $n$ is odd, we have 
\begin{equation*}
{\HMOL} = C(L) \cdot F(L) \cdot G(L) \cdot \prod_{k=1}^{(n+1)/2}\frac{|B_{2k}|}{2k}. 
\end{equation*}

\noindent
(2) When $n$ is even, we have 
\begin{equation*}
{\HMOL} = 
C(L) \cdot F(L) \cdot H(L) \cdot 
\prod_{k=1}^{n/2}\frac{|B_{2k}|}{2k} \cdot (n/2)! \cdot L(n/2+1, \; \chi_{D}).  
\end{equation*}
Here $B_{2k}$ are the Bernoulli numbers and 
$L(s, \chi_{D})=\prod_p(1-\chi_{D}(p)p^{-s})^{-1}$ is 
the Dirichlet $L$-function for the quadratic character $\chi_D$. 
\end{proposition}

\subsection{Common estimates}\label{ssec:initial estimate}

Now let $L$ be an even lattice as in \eqref{condition:L=2U+M}, 
which is assumed through the rest of this section. 
We want to estimate the volume ratio ${\HM}(L, K)$ 
for each type of stably reflective vector $l$. 
Since $K$ contains $U$ by Lemma \ref{K contains U}, 
we can calculate ${\HMOK}$ by replacing $L$ with $K$ in Proposition \ref{HM formula}. 
To specify the dependence on the lattice, 
we shall write the numbers $D, n_{p,j}, s_p,\cdots$ as 
$D(L), n_{p,j}(L), s_p(L),\cdots$ and $D(K), n_{p,j}(K), s_p(K),\cdots$. 

In order to avoid repetition, 
we try to make a common estimate where it seems possible.  
More precisely, 
(1) we first give an estimate that is common for all types of $l$ while leaving some terms untouched, and 
(2) then refine it to the final estimate for each AI,$\cdots$, BIII type. 
This will save the length of the article, 
of course at the cost of (small) overestimate. 
In this \S \ref{ssec:initial estimate} we perform the step (1), separating cases by the parity of $n$. 




\subsubsection{The case of odd $n$} 

By Proposition \ref{HM formula}, ${\HM}(L, K)$ is written as 
\begin{equation*}
\frac{C(K)}{C(L)}\cdot \frac{F(K)}{F(L)}\cdot \frac{H(K)}{G(L)}\cdot \frac{2}{|B_{n+1}|} \cdot
((n+1)/2) ! \cdot L((n+1)/2, \chi_{D(K)}). 
\end{equation*} 
The first term is, by the definition, equal to 
\begin{equation*}
\frac{C(K)}{C(L)} = 
2\pi^{-(n+1)/2} \cdot \left( \frac{|A_K|}{|A_L|} \right)^{n/2+1} \cdot |A_L|^{-1/2} \cdot 
\prod_{p}\frac{C_p(K)}{C_p(L)}. 
\end{equation*}
We can estimate $G(L)$ as 
\begin{equation}\label{eqn:G(L)>}
G(L) \geq \prod_{p||A_L|}(1-p^{-n_p(L)/2}) > \zeta(2)^{-1} 
\end{equation}  
where $n_p(L)\geq{\rm rk}(2U)=4$. 
Similarly, we bound $G(K)$ as  
\begin{equation}\label{eqn:G(K)<}
G(K)   
\leq  (1+2^{-1}) \cdot \prod_{\begin{subarray}{c} p||A_K|, \: p>2 \\ n_p(K)\, even \end{subarray}} (1+p^{-n_p(K)/2})  
<  (9/8) \cdot\zeta(2), 
\end{equation}
where $n_p(K)\geq3$ holds for $p>2$ as can be seen from Table \ref{summary table}. 
By Euler's formula we can evaluate $\zeta(2)=\pi^2/6$. 
We also have 
\begin{eqnarray*}
& &        (H(K)/G(K)) \cdot L((n+1)/2, \chi_{D(K)}) \\ 
& = & \prod_{p||A_K|} (1-p^{-n-1})^{-1} \cdot \prod_{p\nmid |A_K|} (1-\chi_{D(K)}(p)p^{-(n+1)/2})^{-1}  \\ 
& \leq & \zeta((n+1)/2). 
\end{eqnarray*} 
The $F$-term can be bounded as follows. 

\begin{lemma}\label{F-ratio n odd}
We have $F(K)/F(L) \leq 1$ when $n$ is odd. 
\end{lemma}

\begin{proof}
Below we use the convention that 
whenever we write a product $\prod_{k=k_0}^{k_1}a_k$ with $k_0>k_1$, it means $=1$. 
Then we can write 
\begin{equation*}
F(L) = \prod_{p} \prod_{k=[n_p(L)/2]+1}^{(n+1)/2} (1-p^{-2k}),  \quad \; 
F(K) = \prod_{p} \prod_{k=[n_p(K)/2]+1}^{(n+1)/2} (1-p^{-2k}). 
\end{equation*}           
Since $n_p(K) \leq n_p(L)$ for each $p\geq2$,  
every factor of $F(L)$ appears also in $F(K)$. 
\end{proof}

To sum up, we obtain the following intermediate estimate. 

\begin{lemma}\label{initial estimate odd}
When $n$ is odd, we have 
\begin{equation*}
\begin{split}
{\HM}(L, K) 
& <  2^{-3}\cdot \pi^{(-n+7)/2} \cdot |B_{n+1}|^{-1} \cdot ((n+1)/2)! \cdot \zeta((n+1)/2) \\ 
& \times ( |A_K|/|A_L| )^{n/2+1} \cdot |A_L|^{-1/2} \cdot \prod_{p}C_p(K)/C_p(L). 
\end{split}
\end{equation*}
\end{lemma}

\subsubsection{The case of even $n$} 

By Proposition \ref{HM formula}, ${\HM}(L, K)$ is written as 
\begin{equation*}
\frac{C(K)}{C(L)}\cdot \frac{F(K)}{F(L)}\cdot \frac{G(K)}{H(L)}\cdot (n/2)!^{-1} \cdot L(n/2+1, \chi_{D(L)})^{-1}. 
\end{equation*} 
The first term is equal to 
\begin{equation*}
\frac{C(K)}{C(L)} = 
2\pi^{n/2+1} \cdot \left( \frac{|A_K|}{|A_L|} \right)^{n/2+1} \cdot |A_L|^{-1/2} \cdot \prod_{p}\frac{C_p(K)}{C_p(L)}. 
\end{equation*}
The same arguments as \eqref{eqn:G(L)>} and \eqref{eqn:G(K)<} show that 
$G(L) > \zeta(2)^{-1}$ and  
$G(K) <  (9/8)\cdot\zeta(2)$.  
We also have 
\begin{eqnarray*}
(H(L)/G(L)) \cdot L(n/2+1, \chi_{D(L)})  
& \geq & \prod_{p||A_L|} (1-p^{-n-2})^{-1} \cdot \prod_{p\nmid |A_L|} (1+p^{-n/2-1})^{-1}  \\ 
& =  & \zeta(n+2) \cdot \prod_{p\nmid |A_L|} (1-p^{-n/2-1})  \\ 
& \geq & \zeta(n+2) \cdot \zeta(n/2+1)^{-1}.  
\end{eqnarray*} 
We next estimate the $F$-term. 
 
\begin{lemma}\label{F-ratio n even}
We have $F(K)/F(L) < \zeta(n+2)$ when $n$ is even. 
\end{lemma}

\begin{proof}
We use the same convention as in the proof of Lemma \ref{F-ratio n odd}. 
Then $F(L)$ and $F(K)$ can be written as  
\begin{equation*}
F(L) = \prod_{p} \prod_{k=[n_p(L)/2]+1}^{n/2+1} (1-p^{-2k}),  \quad \; 
F(K) = \prod_{p} \prod_{k=[n_p(K)/2]+1}^{n/2} (1-p^{-2k}). 
\end{equation*}           
Since $n_p(K) \leq n_p(L)$ for every $p\geq2$,  
we see that $F(K)/F(L)$ is smaller than $\prod_{p}(1-p^{-n-2})^{-1}$.   
\end{proof} 

To sum up, we obtain

\begin{lemma}\label{initial estimate even}
When $n$ is even, we have 
\begin{equation*}
\begin{split}
{\HM}(L, K) & < 2^{-4}\cdot \pi^{n/2+5} \cdot (n/2)!^{-1} \cdot \zeta(n/2+1)  \\ 
                 & \times ( |A_K|/|A_L| )^{n/2+1} \cdot |A_L|^{-1/2} \cdot \prod_{p}C_p(K)/C_p(L). 
\end{split}
\end{equation*}
\end{lemma}

\subsection{Estimate of type A}\label{ssec:type A estimate}

Let the stably reflective vector $l\in L$ be of type AI or AII. 
We are going to develop the estimate in \S \ref{ssec:initial estimate} in these cases. 
First recall from Table \ref{summary table} that 
\begin{equation}\label{eqn:discri ratio A}
|A_K|/|A_L| = 2^{-1}, \; \; 2 
\end{equation}
in the AI, AII cases respectively. 

Next, notice that $L$ contains $\langle-2\rangle \oplus K$ with index $\leq2$ in both cases. 
Hence for $p>2$ we have the orthogonal splitting 
\begin{equation*}
L \otimes {\Zp} \simeq \langle-2\rangle \oplus (K\otimes {\Zp}). 
\end{equation*}
In particular, we have 
$L_{p,j} \simeq K_{p, j}$ for $j>0$ and hence  
\begin{equation}\label{eqn:Cp type A} 
C_p(L) = C_p(K), \quad p>2. 
\end{equation}
It remains to estimate $C_2(K)/C_2(L)$.

\begin{lemma}\label{C2 ratio A}
We have 
\begin{equation*}\label{eqn:C2 ratio B}
\frac{C_2(K)}{C_2(L)} \leq 
\begin{cases} 
\: \; \; 2^{n+2},           &  *= \textrm{AI}, \\ 
\: 3^{-1}\cdot2^6,       &  *= \textrm{AII}. 
\end{cases}
\end{equation*}
\end{lemma}

\begin{proof}
We first consider the AI case. 
Since $L\simeq\langle-2\rangle\oplus K$, 
we may take Jordan decompositions of $L\otimes{\Z}_2$ and $K\otimes{\Z}_2$ so that 
$L_{2,j} \simeq K_{2,j}$ for $j\ne1$ and  
$L_{2,1} \simeq \langle-1\rangle\oplus K_{2,1}$. 
Then we can extract the following inequalities: 
\begin{equation*}
s_2(L) \leq s_2(K)+1, 
\end{equation*}
\begin{equation*}
s_2'(L) \leq s_2'(K)+1, 
\end{equation*}
\begin{equation}\label{eqn:overestimate 1}
w_2(L)-w_2(K) = \sum_{j>0}n_{2,j}(L) \leq n-2, 
\end{equation}
\begin{equation}\label{eqn:q estimate AI}
q(K)-q(L)  = q_1(K)-q_1(L) \leq n_{2,1}(K)+1-n_{2,1}(L) = 0.  
\end{equation}
In the last inequality $q_1(L)\geq n_{2,1}(L)$ holds because $L_{2,1}$ is odd. 
We also have  
\begin{equation*}
\prod_{j>0}\frac{P_2(n_{2,j}^+(L)/2)}{P_2(n_{2,j}^+(K)/2)} = 
\frac{P_2(n_{2,1}^+(L)/2)}{P_2(n_{2,1}^+(K)/2)} \leq 1,  
\end{equation*}
\begin{equation*}
\prod_{j>0}\frac{E_{2,j}(K)}{E_{2,j}(L)} = 
\frac{E_{2,1}(K)}{E_{2,1}(L)} \cdot \frac{E_{2,2}(K)}{E_{2,2}(L)} \leq 2\cdot\frac{2}{1} = 4.  
\end{equation*}

We next consider the AII case. 
By Lemma \ref{A_K AII} we can take Jordan decompositions of $L\otimes{\Z}_2$ and $K\otimes{\Z}_2$ so that  
$L_{2,j} \simeq K_{2,j}$ for $j>1$ and  
$L_{2,1} \oplus \langle \epsilon \rangle \simeq K_{2,1}$ for some $\epsilon\in{\Z}_2^{\times}$. 
Then we can see the following: 
\begin{equation*}
s_2(L) \leq s_2(K), 
\end{equation*}
\begin{equation*}
s_2'(L) \leq s_2'(K)+1, 
\end{equation*}
\begin{equation}\label{eqn:overestimate 2}
w_2(L)-w_2(K) = -1-\sum_{j>0}n_{2,j}(L) \leq -1-n_{2,1}(L), 
\end{equation}
\begin{equation}\label{eqn:q estimate AII}
q(K)-q(L)  = q_1(K)-q_1(L) \leq q_1(K) \leq n_{2,1}(L)+2,  
\end{equation}
\begin{equation*}
\prod_{j>0}\frac{P_2(n_{2,j}^+(L)/2)}{P_2(n_{2,j}^+(K)/2)} = 
\frac{P_2(n_{2,1}^+(L)/2)}{P_2(n_{2,1}^+(K)/2)} 
\leq \frac{1}{1-2^{-n_{2,1}^+(L)-2}} \leq \frac{4}{3},  
\end{equation*}
\begin{equation*}
\prod_{j>0}\frac{E_{2,j}(K)}{E_{2,j}(L)} = 
\frac{E_{2,1}(K)}{E_{2,1}(L)} \cdot \frac{E_{2,2}(K)}{E_{2,2}(L)} \leq  
\left( \frac{1}{1-2^{-1}} \right)^2 = 4. 
\end{equation*}
\end{proof}

By incorporating \eqref{eqn:discri ratio A}, \eqref{eqn:Cp type A}, Lemma \ref{C2 ratio A} 
with Lemmas \ref{initial estimate odd} and \ref{initial estimate even}, 
we obtain the final estimate for type A.

\begin{proposition}\label{vol estimate A}
Let the stably reflective vector $l\in L$ be of type $*=$ AI or AII. 
We define functions $f_{AI}(n)$ and $f_{AII}(n)$ by 
\begin{equation*}\label{eqn:fAI}
f_{AI}(n) = 
\begin{cases} 
2^{n/2-2}\cdot \pi^{(-n+7)/2} \cdot |B_{n+1}|^{-1} \cdot ((n+1)/2)! \cdot \zeta((n+1)/2)   &  n:  \textrm{odd}, \\ 
2^{n/2-3}\cdot \pi^{n/2+5} \cdot (n/2)!^{-1} \cdot \zeta(n/2+1)                                     & n:  \textrm{even}, 
\end{cases}
\end{equation*}
\begin{equation*}\label{eqn:fAII}
f_{AII}(n) = 3^{-1} \cdot 2^{6} \cdot f_{AI}(n). 
\end{equation*}
Then we have 
\begin{equation*}
{\HM}(L, K) <  f_{*}(n)\cdot |A_L|^{-1/2}. 
\end{equation*}
\end{proposition}

\subsection{Estimate of type B}\label{ssec:type B estimate}

Let the stably reflective vector $l\in L$ be of type BI, BII or BIII. 
In particular, $A_L\simeq{\Z}/D\oplus({\Z}/2)^m$ with $D>2$. 
We continue the estimate of \S \ref{ssec:initial estimate} in these cases. 
First recall from Table \ref{summary table} that 
\begin{equation}\label{eqn:discri ratio B}
|A_K|/|A_L| = 1/D, \; \; 4/D, \; \; 2/D 
\end{equation}
in the BI, BII, BIII cases respectively. 

Notice that $A_K$ is 2-elementary and $L$ contains ${\Z}l \oplus K$ with index $\leq2$ in every case.  
Therefore, for $p|D$ with $p>2$, we have the splitting 
\begin{equation*}\label{eqn:Zp split type B}
L\otimes{\Zp} = {\Z}_pl \oplus (K\otimes{\Zp}) 
\end{equation*}
with $K\otimes{\Zp}$ unimodular and $(l, l)=-D$ or $-2D$. 
In particular, $C_p(K)=1$ for $p>2$. 
We can also explicitly calculate $C_p(L)$ for $p>2$. 
Let $p^{\mu}||D$. 
Then we see that  
$s_p(L) = 2$, 
$w_p(L) = \mu$ and  
\begin{equation*}
P_p([n_{p,\mu}(L)/2]) = 1+\chi(L_{p,\mu})p^{-n_{p,\mu}(L)/2} = 1. 
\end{equation*}
Hence 
$C_p(L) = 2^{-1}\cdot p^{-\mu}$. 
If we write $D=2^{\nu}\cdot D_{\ne2}$ with $D_{\ne2}$ odd, then  
\begin{equation}\label{eqn:Codd type B}
\prod_{p>2} C_p(K)/C_p(L) = 2^{\omega(D_{\ne2})} \cdot D_{\ne2} 
\end{equation}
where $\omega(D_{\ne2})$ is the number of prime divisors of $D_{\ne2}$. 
The remaining term to estimate is $C_2(K)/C_2(L)$.

\begin{lemma}\label{C2 ratio B}
Let $2^{\nu}||D$. 
For type $*=$ BI, BII we have 
\begin{equation*}\label{eqn:C2 ratio B}
\frac{C_2(K)}{C_2(L)} \leq 
\begin{cases} 
\: 2^{m+3+\nu},           &  *= \textrm{BI}, \\ 
\: 3^{-1}\cdot2^{\nu+4},       &  *= \textrm{BII}, 
\end{cases}
\end{equation*}
and for type BIII we have 
$C_2(K)/C_2(L)=2^{-1}$. 
\end{lemma} 

\begin{proof}
In the BIII case, $L\otimes{\Z}_2$ is unimodular so that $C_2(L)=1$. 
Since $A_K\simeq{\Z}/2$, it is easy to calculate that $C_2(K)=2^{-1}$. 
The calculations in the BI, BII cases are similar to Lemma \ref{C2 ratio A}: 
we can compare appropriate Jordan decompositions of $L\otimes{\Z}_2$ and $K\otimes{\Z}_2$ to estimate 
\begin{equation*}
s_2(L) - s_2(K) \leq 1,  
\end{equation*}
\begin{equation*}
s_2'(L) - s_2'(K) \leq 2,   
\end{equation*}
\begin{equation*}
w_2(L)-w_2(K) = 
\begin{cases} 
\: \; \; m+\nu,       &  *= \textrm{BI}, \\ 
\:    \nu-m-3,        &  *= \textrm{BII}, 
\end{cases}
\end{equation*}
\begin{equation}\label{eqn:overestimate 3}
q(K) - q(L) \leq 
\begin{cases} 
\: \; \;    -1,       &  *= \textrm{BI}, \\ 
\:      m+1,        &  *= \textrm{BII}, 
\end{cases} 
\end{equation}
\begin{equation*}
\prod_{j>0} \frac{P_2(n_{2,j}^+(L)/2)}{P_2(n_{2,j}^+(K)/2)} 
= \frac{P_2(n_{2,1}^+(L)/2)}{P_2(n_{2,1}^+(K)/2)} \leq 
\begin{cases} 
\: \; \;    1,                    &  *= \textrm{BI}, \\ 
\:  (1-2^{-2})^{-1},        &  *= \textrm{BII}, 
\end{cases}  
\end{equation*}
\begin{equation*}
\prod_{j>0}E_{2,j}(K)/E_{2,j}(L) \leq 2. 
\end{equation*}
\end{proof}

Substituting \eqref{eqn:discri ratio B}, \eqref{eqn:Codd type B} and Lemma \ref{C2 ratio B} with $m+3\leq n$ into 
Lemmas \ref{initial estimate odd} and \ref{initial estimate even}, 
we obtain  

\begin{proposition}\label{vol estimate B}
Let the stably reflective vector $l\in L$ be of type $*=$ BI, BII or BIII. 
We define functions $g_{BI}(n)$, $g_{BII}(n)$ and $g_{BIII}(n)$ by 
\begin{equation*}\label{eqn:fBI}
g_{BI}(n) = 
\begin{cases} 
2^{n-4}\cdot \pi^{(-n+7)/2} \cdot |B_{n+1}|^{-1} \cdot ((n+1)/2)! \cdot \zeta((n+1)/2)   &  n:  \textrm{odd}, \\ 
2^{n-5}\cdot \pi^{n/2+5} \cdot (n/2)!^{-1} \cdot \zeta(n/2+1)                                     & n:  \textrm{even}. 
\end{cases}
\end{equation*}
\begin{equation*}\label{eqn:fBII}
g_{BII}(n) = 3^{-1} \cdot 2^6 \cdot g_{BI}(n), 
\end{equation*}
\begin{equation*}\label{eqn:fBII}
g_{BIII}(n) = 2^{1-n/2} \cdot g_{BI}(n), 
\end{equation*}
with $g_{BIII}(n)$ defined only for even $n$. 
Then  we have 
\begin{equation*}
{\HM}(L, K) < g_{*}(n)\cdot D^{-n/2} \cdot 2^{\omega(D)} \cdot |A_L|^{-1/2}. 
\end{equation*}
\end{proposition}

If we further make the (over)estimate 
\begin{equation}\label{eqn:overestimate 4}
D^{-n/2} \cdot 2^{\omega(D)} < D^{-n/2+1} \leq 3^{-n/2+1},  
\end{equation}
we can also obtain an expression only in terms of $n$ and $|A_L|$. 

\begin{corollary}\label{estimate B by n AL}
We set $f_{*}(n)=3^{-n/2+1}\cdot g_{*}(n)$ for $*=$ BI, BII, BIII. 
Then 
\begin{equation*}
{\HM}(L, K)  <  f_{*}(n)\cdot |A_L|^{-1/2}. 
\end{equation*}
\end{corollary}

This completes the proof of Proposition \ref{vol ratio}.

\begin{remark}\label{effective bigness}
The estimates in \S \ref{sec:sum over branch div} and \S \ref{sec:vol ratio} 
are derived uniformly so that they would be overrating for many lattices. 
In order to make Theorem \ref{bigness} as effective as possible, 
one should remember where to improve them. 
In particular, notice that 
\begin{itemize}
\item The bounds  
\eqref{eqn:overestimate 5}, \eqref{eqn:overestimate 1}, \eqref{eqn:q estimate AI}, \eqref{eqn:overestimate 2}, 
\eqref{eqn:q estimate AII}, \eqref{eqn:overestimate 3}, \eqref{eqn:overestimate 4}  
and the inequality $m+3\leq n$ substituted in Lemmas \ref{estimate ratio BII} and \ref{C2 ratio B} 
might be sometimes too overestimating. 
Some of them could be largely reduced when the class of lattices is specified. 
\item It is also useful to observe that 
$e_{AII}(n)f_{AII}(n)$ (or its improvement) is much greater than other $e_{*}(n)f_{*}(n)$ as $n$ grows.  
Thus it gives the main term in the estimate of \eqref{eqn:(3.3)} if $n$ is not so small. 
This is first observed in \cite{G-H-S3} in a special case. 
\item Actually, for majority of lattices we have only branch divisor of type AII. 
Reflections of type B take place only for special type of lattices. 
\end{itemize}
For instance, if we restrict to maximal lattices $L$, we can reduce the estimate to the approximate form 
\begin{equation*}
\eqref{eqn:(3.3)}  \leq [ \, C\cdot(4\pi e/n)^{(n+1)/2} + \textrm{"error term"} \, ] \cdot |A_L|^{-1/2} 
\end{equation*}
where $C=(\pi/2)^4\cdot e^{-1/2}$, with the aid of Stirling's formula. 
\end{remark}


\appendix

\section{Quasi-cyclic forms}\label{sec:appendix}

By a \textit{finite quadratic form} we mean a finite abelian group $A$ endowed with a quadratic form $A\to{\Q}/2{\Z}$ 
which we assume to be nondegenerate throughout. 
We say that a finite quadratic form is \textit{quasi-cyclic} if any isotropic subgroup of it is cyclic. 
This class of quadratic forms obviously contains the anisotropic ones. 
In this appendix we classify quasi-cyclic forms (\S \ref{ssec:classify quasi-cyclic}), 
and present an ``economic'' method to produce such a form out of a given quadratic form (\S \ref{ssec:quasi-cyclification}).  
This section may be read independently of other sections. 
The results are used in \S \ref{ssec:reduction}, \S \ref{decomposition} and \S \ref{ssec:prf eff}, 
where even lattices with quasi-cyclic discriminant form play a central role.

Here is our basic reduction: 
if $A = \oplus_{p} A_p $
is the decomposition into $p$-parts of a finite quadratic form $A$, 
it is easy to see that $A$ is quasi-cyclic if and only if $A_p$ is so for each $p$. 
Thus we may restrict our attention to quasi-cyclic forms on $p$-groups.

\subsection{Classification}\label{ssec:classify quasi-cyclic}

\subsubsection{The odd prime case}

Let $p>2$. 
We can (and do) identify ${\Q}/2{\Z}$-valued quadratic forms on $p$-groups with 
${\Q}/{\Z}$-valued symmetric bilinear forms. 
If $\varepsilon\in{\Zpuni}$ and $k\in{\N}$, 
we write $\langle \varepsilon/p^k\rangle$ for the form on ${\Z}/p^k$ in which the natural generator has norm $\varepsilon/p^k$. 
It is well-known that quadratic forms on $p$ -groups are orthogonal direct sums of these cyclic forms (\cite{Wa}). 
Recall also that anisotropic forms on $p$-groups are either   
\begin{itemize}
\item $\langle \varepsilon/p\rangle$,  
\item $\langle 1/p\rangle \oplus \langle -\varepsilon_0/p\rangle$ where $\varepsilon_0 \notin ({\Zpuni})^2$. 
\end{itemize}
In particular, they are $p$-elementary of length $\leq2$. 
Quasi-cyclic forms can be classified as follows. 

\begin{proposition}\label{q-cyclic p>2}
A quasi-cyclic form on a $p$-group with $p>2$ is isometric to one of the following types of quadratic form. 
\begin{enumerate}
\item $\langle \varepsilon/p^k\rangle \oplus A'$ with $k>1$ and $A'$ anisotropic; 
\item $(\langle 1/p\rangle \oplus \langle -1/p\rangle)^l \oplus A'$ with $l\leq1$ and $A'$ anisotropic.  
\end{enumerate}
Conversely, these forms are always quasi-cyclic. 
\end{proposition}

\begin{proof}
Let $A$ be a quasi-cyclic form on a $p$-group. 
When $A$ is not $p$-elementary, we have an orthogonal splitting 
$A\simeq \langle \varepsilon/p^k\rangle\oplus A'$ with $k>1$. 
Since $\langle \varepsilon/p^k\rangle$ contains a nontrivial isotropic subgroup, $A'$ must be anisotropic. 
When $A$ is $p$-elementary, it can be written in the form 
$(\langle 1/p\rangle \oplus \langle -1/p\rangle)^l \oplus A'$ for some $l\geq0$ and anisotropic $A'$. 
The quasi-cyclic condition implies $l\leq1$. 

Conversely, we show that the above forms are quasi-cyclic. 
This is trivial in the case (2). 
Consider $A=\langle \varepsilon/p^k\rangle \oplus A'$ with $k>1$ and $A'$ anisotropic. 
When $k$ is even, every isotropic element of $A$ must be contained in $\langle \varepsilon/p^k\rangle$ 
because nonzero elements of $A'$ have norm in $p^{-1}{\Zpuni}$ 
while elements of $\langle \varepsilon/p^k\rangle$ never have such norm. 
Next let $k=2N+1$ be odd. 
Choose a generator $x_0$ of $\langle \varepsilon/p^k\rangle$. 
By a similar consideration, isotropic elements of $A$ are either 
(i) contained in $\langle p^{N+1}x_0\rangle$ or 
(ii) of the form $x_1+x_2$ with $x_1$ generating $\langle p^Nx_0\rangle$ and $0\ne x_2\in A'$. 
Assume to the contrary that 
there were two isotropic elements $x, x'$ with $(x, x')=0$ and $x\not\in \langle x'\rangle$, $x'\not\in \langle x\rangle$. 
Then both $x$ and $x'$ should be of type (ii), 
so we write $x=x_1+x_2$ and $x'=x'_1+x'_2$ as above. 
Multiplying $x$ by an element of ${\Zpuni}$, we may assume $x_1=x_1'$. 
From the relation $(x, x)=(x', x')=(x, x')=0$, we can calculate that $(x_2-x_2', x_2-x_2')=0$. 
But since $A'$ is anisotropic, we must have $x_2=x_2'$. 
This is a contradiction. 
\end{proof}

\begin{corollary}\label{q-cyclic length p>2}
Quasi-cyclic forms on $p$-groups with $p>2$ have length $\leq4$. 
\end{corollary}

\subsubsection{The $p=2$ case}

We keep the notation $\langle \varepsilon/2^k\rangle$ for quadratic forms on ${\Z}/2^k$ where $\varepsilon\in{\Z}_2^{\times}$. 
We shall especially write $a=\langle1/2\rangle$ and $b=\langle-1/2\rangle$. 
We also denote by $u_k$, $v_k$ the quadratic forms on $({\Z}/2^k)^2$ given by Gram matrices 
$2^{-k} \begin{pmatrix}  0 & 1 \\ 1 & 0 \end{pmatrix}$, 
$2^{-k} \begin{pmatrix}  2 & 1 \\ 1 & 2 \end{pmatrix}$ respectively. 
It is known (\cite{Wa}) that quadratic forms on 2-groups are orthogonal direct sums of 
$u_k$, $v_k$ and cyclic forms $\langle \varepsilon/2^k\rangle$, 
with various relations among these generators. 
In particular, they can be written as 
\begin{equation}\label{eqn:normal form p=2}
\bigoplus_{k>0}u_k^l\oplus v_k^m\oplus A_k 
\end{equation}
where $m\leq1$ and $A_k$ is a sum of at most two cyclic forms of order $2^k$. 
Anisotropic forms are exactly the following types of quadratic form: 
\begin{itemize}
\item $\langle \varepsilon/4\rangle \oplus \langle \varepsilon'/2\rangle$, 
\item $\langle \varepsilon/4\rangle$, 
\item the 2-elementary forms $a^l, b^l$ ($l\leq3$) and $v_1$. 
\end{itemize}

Now quasi-cyclic forms can be classified as follows. 

\begin{proposition}\label{q-cyclic p=2}
A quasi-cyclic form on a 2-group  is isometric to one of the following types of quadratic form.  
\begin{enumerate}
\item $\langle \varepsilon/2^k\rangle \oplus A'$ with $k>1$ and $A'$ anisotropic; 
\item $\langle \varepsilon/4 \rangle \oplus A'$ with $A'$ being 2-elementary of length $\leq3$;  
\item the 2-elementary forms $u_1$, $v_1$, $u_1\oplus v_1$ and $a^l\oplus b^m$ with $l+m\leq5$ and ${\min}(l, m)\leq1$.  
\end{enumerate}
Conversely, these forms are always quasi-cyclic. 
\end{proposition}

The cases (1) and (2) have overlap, but to keep the presentation simple we do not care about this. 
Note that forms $\langle \varepsilon/2\rangle \oplus A'$ with $A'$ anisotropic are quasi-cyclic as well 
by the cases (2) and (3).

\begin{proof} 
Let $A$ be a quasi-cyclic form on a 2-group of exponent $2^k$. 
When $k=1$, a direct calculation with \eqref{eqn:normal form p=2} shows that $A$ must be one of the forms in (3). 
In case $k\geq2$, $A$ cannot contain neither $u_k$ nor $v_k$ 
because they contain isotropic subgroups isomorphic to $({\Z}/2)^2$. 
Hence we can write $A$ as $A=\langle \varepsilon/2^k\rangle \oplus A'$. 
When $k\geq3$, $\langle \varepsilon/2^k\rangle$ has a nontrivial isotropic element so that $A'$ must be anisotropic. 
When $k=2$, either 
(i) $A'=\langle \varepsilon'/4\rangle \oplus A''$ with $A''$ 2-elementary or 
(ii) $A'$ itself is 2-elementary. 
In the first case, the subgroup 
$({\Z}/2)^2\subset \langle \varepsilon/4\rangle \oplus \langle \varepsilon'/4\rangle$ 
is isotropic modulo ${\Z}$. 
It follows that $A''$ cannot have an element of norm $\in{\Z}$, 
and hence has length $\leq1$. 
Similarly, in case (ii), $A'$ cannot contain a subgroup of length $2$ that is isotropic modulo ${\Z}$. 
Therefore $A'$ has length $\leq3$. 

Conversely, it is easy to check that the forms in (2) and (3) are indeed quasi-cyclic. 
It then remains to show that forms 
$A=\langle \varepsilon/2^k\rangle \oplus A'$ with $k\geq2$ and $A'$ anisotropic are quasi-cyclic. 
Suppose to the contrary that $A$ has two isotropic elements 
$x$, $x'$ with $(x, x')=0$ and $x\not\in\langle x'\rangle$, $x'\not\in\langle x\rangle$. 
Write $x=x_1+x_2$ and $x'=x_1'+x_2'$ according to the given decomposition $A=\langle \varepsilon/2^k\rangle \oplus A'$. 
In case $A'$ is 2-elementary, the same argument as the proof of Proposition \ref{q-cyclic p>2} can be applied. 
Next let $A'\simeq \langle \varepsilon'/4\rangle \oplus A''$ with $A''$ being 2-elementary. 
When $k$ is odd, $x_2$ and $x_2'$ should be either $0$ or have norm $1/2$ mod ${\Z}$ and order $2$. 
This time the same proof works again. 
When $k$ is even, we can still argue that 
$x_2$ and $x_2'$ should be either 
(i) both equal to the unique element of norm $1$ in $A'$ or 
(ii) both of order $4$. 
This ensures that $\langle x_1\rangle=\langle x_1'\rangle$, 
and again we can arrive at a contradiction. 
\end{proof}

\begin{corollary}\label{q-cyclic length p=2}
Quasi-cyclic forms on 2-groups have length $\leq5$. 
\end{corollary}

By Corollaries \ref{q-cyclic length p>2} and \ref{q-cyclic length p=2}, 
quasi-cyclic forms have length $\leq5$. 
It follows from \cite{Ni} that 

\begin{corollary}\label{q-cyclic contain 2U}
If an even lattice of signature $(2, n)$ with $n\geq8$ has quasi-cyclic discriminant form, 
it contains $2U$. 
\end{corollary}

\subsection{Quasi-cyclic overlattices}\label{ssec:quasi-cyclification}

Let $A$ be a given finite quadratic form. 
Our aim below is to find an isotropic subgroup $G\subset A$ such that 
(i) the quadratic form induced on $G^{\perp}/G$ is quasi-cyclic and that 
(ii) the ``size'' of $G^{\perp}/G$ is kept as large as possible. 
Precisely, we will adopt exponent as the measure of ``size''. 
Of course we could take $G$ so that $G^{\perp}/G$ is anisotropic, but then the ``size'' would be reduced too much. 
Due to this, the argument of \S \ref{ssec:reduction} for fixed (small) $n$ would not work 
if we use only maximal overlattices of given lattices. 
This is why we consider quasi-cyclic lattices in the reduction process: 
we have plenty of such lattices to meet (ii), 
and with them we can still obtain cusp forms by the Jacobi lifting. 

As usual, we shall localize the problem and work at each prime $p$. 
First consider the case $p>2$. 

\begin{lemma}\label{q-cyclic overlattice p>2}
Let $A$ be a quadratic form on a $p$-group with $p>2$. 
Then there exists an isotropic subgroup $G\subset A$ such that 
$G^{\perp}/G$ is quasi-cyclic and has the same exponent as $A$. 
\end{lemma}

\begin{proof}
Let $p^k$ be the exponent of $A$. 
We can find an orthogonal splitting $A=\langle \varepsilon/p^k\rangle \oplus A'$. 
Choose an isotropic subgroup $G$ of $A'$ that is maximal inside $A'$. 
Then $(G^{\perp}\cap A')/G$ is an anisotropic form. 
By Proposition \ref{q-cyclic p>2}, 
\begin{equation*}
(G^{\perp}\cap A)/G \simeq \langle \varepsilon/p^k\rangle \oplus (G^{\perp}\cap A')/G
\end{equation*}
is quasi-cyclic. 
It obviously has the same exponent as $A$. 
\end{proof}

The case $p=2$ is a bit more complicated. 

\begin{lemma}\label{q-cyclic overlattice p=2}
Let $A$ be a quadratic form on a 2-group of exponent $2^k$. 
We define $k'$ by $k'=k$ if 
$A$ contains a cyclic form $\langle \varepsilon/2^k\rangle$ for some $\varepsilon\in{\Z}_2^{\times}$ 
or if $A$ is 2-elementary, 
and $k'=k-1$ otherwise. 
Then there exists an isotropic subgroup $G\subset A$ such that $G^{\perp}/G$ is quasi-cyclic of exponent $2^{k'}$. 
\end{lemma}

\begin{proof}
When a cyclic form $\langle \varepsilon/2^k\rangle$ splits off $A$, 
the same construction as in the case $p>2$ works. 
The case $A$ is 2-elementary is easy. 
Then let $A$ be of the form $u_k\oplus A'$ or $v_k\oplus A'$ 
with $k>1$ and with $A'$ not containing $\langle \varepsilon/2^k\rangle$. 
We will only study the first case; the second case can be dealt with similarly. 
Let $e, f$ be the standard basis of $u_k$. 
Consider the element 
\begin{equation*}
x = 2^{[(k+1)/2]}(e+f). 
\end{equation*}
This is isotropic, and $x^{\perp}\cap u_k$ is generated by $e-f$ and $2^{[k/2]-1}(e+f)$. 
In particular, 
\begin{equation*}
x^{\perp}\cap u_k/\langle x\rangle \simeq 
\begin{cases}
\langle -1/2^{k-1}\rangle \oplus \langle 1/2\rangle \quad k:\textrm{even}, \\ 
\langle -1/2^{k-1}\rangle \oplus \langle 1/4\rangle \quad k:\textrm{odd}. 
\end{cases}
\end{equation*}
Now consider the quadratic form 
$\bar{A}:= x^{\perp}\cap A/\langle x\rangle$. 
By the above calculation $\bar{A}$ is isometric to 
$\langle -1/2^{k-1}\rangle \oplus \langle 1/2^l\rangle\oplus A'$ with $l=1$ or $2$. 
We take an isotropic subgroup $\bar{G}$ of $\langle 1/2^l\rangle\oplus A'$ 
that is maximal inside $\langle 1/2^l\rangle\oplus A'$. 
The quadratic form $(\bar{G}^{\perp}\cap\bar{A})/\bar{G}$ 
is a sum of $\langle -1/2^{k-1}\rangle$ and an anisotropic form. 
Hence it is quasi-cyclic by Proposition \ref{q-cyclic p=2}, and has exponent $2^{k-1}$. 
We then define $G\subset A$ to be the inverse image of $\bar{G}$ by the projection $x^{\perp}\cap A\to\bar{A}$. 
\end{proof}

By Lemmas \ref{q-cyclic overlattice p>2} and \ref{q-cyclic overlattice p=2}, we obtain the following. 

\begin{proposition}\label{q-cyclic overlattice} 
Let $A$ be a finite quadratic form of exponent $D$. 
Then there exists an isotropic subgroup $G$ of $A$ such that 
the quadratic form induced on $G^{\perp}/G$ is quasi-cyclic and has exponent $D$ or $D/2$. 
\end{proposition}


\vspace{0.4cm}
\noindent
\textbf{Acknowledgement.} 
The author is indebted to Professors S.~Kond\=o and K.~Yoshikawa for their valuable comments and encouragement.



\begin{thebibliography}{99}


\bibitem{Bo}Borcherds, R. 
\textit{The Gross-Kohnen-Zagier theorem in higher dimensions.} 
Duke Math. J. \textbf{97} (1999), no. 2, 219--233. 

\bibitem{Bori}Borisov, L.~A.
\textit{A finiteness theorem for subgroups of ${\rm Sp}(4,{\Z})$.}  
J. Math. Sci. \textbf{94} (1999), no. 1, 1073--1099. 
 
\bibitem{Br}Bruinier, J.~H. 
\textit{On the rank of Picard groups of modular varieties attached to orthogonal groups.}  
Compositio Math. \textbf{133} (2002), no. 1, 49--63. 


\bibitem{E-S}Eholzer, W; Skoruppa, N.-P. 
\textit{Modular invariance and uniqueness of conformal characters.} 
Comm. Math. Phys. \textbf{174} (1995), no.1, 117--136.


\bibitem{Gr}Gritsenko, V. 
\textit{Modular forms and moduli spaces of abelian and K3 surfaces.} 
St. Petersburg Math. J. \textbf{6} (1995), no. 6, 1179--1208. 

\bibitem{Gr2}Gritsenko, V. 
\textit{Reflective modular forms in algebraic geometry.}
arXiv:1005.3753. 

\bibitem{G-H}Gritsenko, V.; Hulek, K.
\textit{Uniruledness of orthogonal modular varieties.}
to appear in J. Algebraic Geom.  

\bibitem{G-H-S1}Gritsenko, V.~A.; Hulek, K.; Sankaran, G.~K. 
\textit{The Kodaira dimension of the moduli of K3 surfaces.}
Invent. Math. \textbf{169} (2007), no. 3, 519--567. 

\bibitem{G-H-S2}Gritsenko, V.; Hulek, K.; Sankaran, G.~K.
\textit{The Hirzebruch-Mumford volume for the orthogonal group and applications.} 
Doc. Math. \textbf{12} (2007), 215--241. 

\bibitem{G-H-S3}Gritsenko, V.; Hulek, K.; Sankaran, G.~K.
\textit{Hirzebruch-Mumford proportionality and locally symmetric varieties of orthogonal type.} 
Doc. Math. \textbf{13} (2008), 1--19. 

 
\bibitem{K-K}Kawauchi, A; Kojima, S. 
\textit{Algebraic classification of linking pairings on 3-manifolds.} 
Math. Ann. \textbf{253} (1980), no. 1, 29--42. 

\bibitem{Ki}Kitaoka, Y. 
\textit{Arithmetic of quadratic forms.}
Cambridge University Press, 1993. 

\bibitem{Ma}Ma, S. 
\textit{E8 lattice and the Kodaira dimension of orthogonal modular varieties.} 
arXiv:1306.2129. 

\bibitem{Mu}Mukai, S. 
\textit{Polarized K3 surfaces of genus thirteen.} 
Moduli spaces and arithmetic geometry, 315--326, 
Adv. Stud. Pure Math., \textbf{45}, Math. Soc. Japan, 2006. 

\bibitem{Ni}Nikulin, V.V. 
\textit{Integral symmetric bilinear forms and some of their applications.}
Math. USSR Izv. \textbf{14} (1980), 103--167. 

\bibitem{Sc}Scattone, F. 
\textit{On the compactification of moduli spaces for algebraic K3 surfaces.} 
Mem. Amer. Math. Soc. \textbf{70} (1987), no.~374. 

\bibitem{Sk}Skoruppa, N.~P. 
\textit{\"Uber den Zusammenhang zwischen Jacobiformen und Modulformen halbganzen Gewichts.} 
Dissertation, Universit\"at Bonn, 1984.  

\bibitem{Ta}Tai, Y.-S. 
\textit{On the Kodaira dimension of the moduli space of abelian varieties.} 
Invent. Math. \textbf{68} (1982), no. 3, 425--439. 

\bibitem{Wa} Wall, C. T. C.
\textit{Quadratic forms on finite groups, and related topics.} 
Topology \textbf{2} (1963) 281--298. 


\end{thebibliography}
\end{document}